\documentclass[11pt,oneside]{amsart}

\usepackage{amssymb}
\usepackage{amsmath}
\usepackage{amsthm}
\usepackage[all]{xy}
\usepackage{comment}

\usepackage[mathscr]{eucal}

\theoremstyle{plain}
\newtheorem{thm}{Theorem}[section]
\newtheorem{prop}[thm]{Proposition}
\newtheorem{lem}[thm]{Lemma}
\newtheorem{cor}[thm]{Corollary}

\newtheorem{clm}[thm]{Claim}

\theoremstyle{definition}
\newtheorem{defn}[thm]{Definition}

\theoremstyle{remark}
\newtheorem{rem}[thm]{Remark}

\DeclareMathOperator{\Pic}{Pic}

\DeclareMathOperator{\Nef}{Nef}
\DeclareMathOperator{\Eff}{Eff}
\DeclareMathOperator{\Mov}{Mov}

\DeclareMathOperator{\Spec}{Spec}

\DeclareMathOperator{\rank}{rank}

\DeclareMathOperator{\id}{id}

\DeclareMathOperator{\Quot}{Quot}

\DeclareMathOperator{\Hom}{Hom}
\DeclareMathOperator{\Mor}{Mor}
\DeclareMathOperator{\Sym}{Sym}
\DeclareMathOperator{\Gr}{Gr}
\DeclareMathOperator{\GL}{GL}
\DeclareMathOperator{\coker}{coker}
\DeclareMathOperator{\pr}{pr}

\def\N{\mathbb{N}}
\def\Z{\mathbb{Z}}
\def\Q{\mathbb{Q}}
\def\R{\mathbb{R}}
\def\C{\mathbb{C}}
\def\A{\mathbb{A}}

\def\r+{\mathbb{R}_{\geq 0}}

\def\r+{{\R}_{\geq 0}}
\def\q+{{\Q}_{\geq 0}}
\def\P{\mathbb{P}}
\def\arw{\rightarrow}
\def\*c{\C^{\times}}

\def\A{\mathbb {A}}

\def\C{\mathbb {C}}

\def\G{\mathbb {G}}

\def\N{\mathbb {N}}

\def\P{\mathbb {P}}
\def\Q{\mathbb {Q}}
\def\R{\mathbb {R}}

\def\V{\mathbb {V}}

\def\Z{\mathbb {Z}}

\newcommand{\cala}{\mathcal {A}}
\newcommand{\calb}{\mathcal {B}}

\newcommand{\cald}{\mathcal {D}}
\newcommand{\cale}{\mathcal {E}}
\newcommand{\calf}{\mathcal {F}}

\newcommand{\calh}{\mathcal {H}}

\newcommand{\calk}{\mathcal {K}}

\newcommand{\calo}{\mathcal {O}}

\newcommand{\calq}{\mathcal {Q}}

\makeatletter
  
  \@addtoreset{equation}{section}
\makeatother

\begin{document}

\title[On the space of parametrized rational curves in Grassmannians]
{On birational geometry of the space of parametrized rational curves in Grassmannians}
\author[A.~Ito]{Atsushi~Ito}
\address{Department of Mathematics, Faculty of Science, Kyoto University,
Kyoto 606-8502, Japan}
\email{aito@math.kyoto-u.ac.jp}

\begin{abstract}
In this paper,
we study the birational geometry of the Quot schemes of trivial bundles on $\P^1$
by constructing small $\Q$-factorial modifications of the Quot schemes as suitable moduli spaces.
We determine all the models which appear in the minimal model program on the Quot schemes.
As a corollary,
we show that the Quot schemes are Mori dream spaces and log Fano.
\end{abstract}


\maketitle

\section{Introduction}\label{section_intro}

Let $V$ be an $n$-dimensional vector space over an algebraically closed field $k$.
For fixed integers $ d \geq 0$ and $0 \leq  r \leq n-1$,
the moduli space $R^{\circ} := \Mor_d (\P^1, \G)$
parametrizes all morphisms $\P^1 \arw  \G$ of degree $d$,
where $\G$ is the Grassmannian of $r$-dimensional quotient spaces of $V$.
Such morphism $\P^1 \arw  \G$
corresponds to a locally free quotient sheaf of $V_{\P^1}:= V \otimes \calo_{\P^1} $ of rank $r$ and degree $d$.
Thus we can compactify $R^{\circ}$ by the Quot scheme
$R$ which parametrizes
all rank $r$, degree $d$ quotient sheaves of $V_{\P^1} $.

In \cite{St},
Str\o mme proved many properties of $R$.
For example,
he showed that $R$ is an irreducible rational smooth projective variety of dimension $nd + r(n-r)$.
He also determined $\Pic(R)$ and the nef cone $\Nef(R) \subset N^1(R)_{\R}:= (\Pic (R) / \equiv ) \otimes \R$ as follows.

If $r=n-1$ or $d=0$,
$R$ is $ \P \big( V^{\vee} \otimes H^0(\calo_{\P^1}(d))^{\vee} \big)$ or the Grassmannian $\G$ respectively.
Hence $\Pic (R) \cong \Z$ and $\Nef(R)$ is spanned by the ample generator.

If $0 \leq r \leq n-2$ and $d \geq 1$,
Str\o mme showed that there exist base point free line bundles $\alpha, \beta$ on $R$ such that
$\Pic (R) = \Z \alpha \oplus \Z \beta$ and $\Nef(R)= \r+ \alpha + \r+ \beta$.

On the other hand,
Jow \cite{Jo} determined the effective cone $\Eff(R)$ of $R$
by constructing two effective divisors which span $\Eff(R)$.
Venkatram \cite{Ve} determined the movable cone $\Mov(R)$ of $R$, and the stable base locus decomposition of $\Eff(R)$.

\vspace{2mm}
Birational geometry of moduli spaces is studied in many papers, \cite{ABCH}, \cite{Ch}, \cite{Ha}, etc.
The purpose of this paper is to investigate the birational geometry of $R$
by constructing small $\Q$-factorial modifications of $R$ as suitable moduli spaces.
Recall the definition of small $\Q$-factorial modifications.

\begin{defn}[{\cite[1.8 Definition]{HK}}]\label{def_SQM}
By a \textit{small $\Q$-factorial modification (SQM)} of a projective variety $X$,
we mean a birational map $f : X \dashrightarrow X'$
with $X'$ projective, normal, and $\Q$-factorial,
such that $f$ is an isomorphism in codimension one.
We note that such $f$ induces an isomorphism $f^* : N^1(X')_{\R} \arw N^1(X)_{\R}$ by pullbacks of divisors
if $X$ is $\Q$-factorial.
\end{defn}

\vspace{2mm}
We can find SQMs of $R$ as moduli spaces parametrizing following objects.

For a scheme $T$,
we denote the second projection $ \P^1 \times T \arw T$ by $\pi_T$, or simply $\pi$.
For a coherent sheaf $\calf$ on $\P^1 \times T$ and $m \in \Z$,
we denote $\calf \otimes {p_1}^* \calo_{\P^1}(m)$ by $\calf (m)$,
where $p_1 : \P^1 \times T \arw \P^1$ is the first projection.

\begin{defn}\label{def_cond_star}
Fix integers $0 \leq r \leq n-1 $, $d \geq 0$, and $m \geq \lceil d/s \rceil$ for $s:=n-r$.
For a locally noetherian $k$-scheme $T$,
a morphism $\iota : \cale \arw V_{\P^1 \times T}:=V \otimes _k \calo_{\P^1 \times T}$ of sheaves on $\P^1 \times T$ satisfies condition $(\star_m)$
if it satisfies the following three conditions:
\begin{itemize}
\item[i)] $\cale$ is locally free of rank $s$ and the degree of $\cale |_{\P^1 \times \{t\}}$ is $-d$ for any $t \in T$,
\item[ii)] $R^1 \pi_* (\cale(m-1)) =0$,
or equivalently,
$H^1(\cale (m-1) |_{\P^1 \times \{t\}}) =0$ for any $t \in T$,
\item[iii)] the induced map $\pi_* (\cale(m)) \otimes k(t) \arw \pi_* (V_{\P^1 \times T}(m)) \otimes k(t) =V \otimes H^0(\calo_{\P^1}(m))$ is injective for any $t \in T$.
\end{itemize}

For $\iota : \cale \arw V_{\P^1 \times T}$ and $\iota' : \cale' \arw V_{\P^1 \times T}$ which satisfy $(\star_m)$,
we say that $\iota: \cale \arw V_{\P^1 \times T} $ and $\iota' : \cale' \arw V_{\P^1 \times T}$ are equivalent
if and only if there exists an isomorphism $\cale \arw \cale'$ such that
\[
\xymatrix{
\cale \ar[r]^(0.43){\iota} \ar[d] & V_{\P^1 \times T} \ar@{=}[d] \\
\cale' \ar[r]^(0.43){\iota'} & V_{\P^1 \times T}\\
 }
\]
is commutative.
We denote by $[ \cale \arw V_{\P^1 \times T}]$ the equivalence class of $\cale \arw V_{\P^1 \times T}$.
\end{defn}

\begin{thm}\label{thm_moduli}
Let $n,r,d$, and $s$ be as in Definition \ref{def_cond_star}.
For each $m \geq \lceil d/s \rceil$,
there exists a smooth projective variety $R_m$
which is the fine moduli space parametrizing equivalence classes $ [\cale \arw V_{\P^1 \times T}]$
which satisfy $(\star_m)$ for each locally noetherian $k$-scheme $T$.
Furthermore,
there exists a natural birational map $\tilde{g}_m : R \dashrightarrow R_m$.
\end{thm}

\begin{rem}\label{rem_same}
For each $ m \geq d$,
Str\o mme \cite{St} defined a subscheme $R_m$ (he denoted it by $Z$ simply) in the direct product of two Grassmannians,
and constructed an isomorphism $\tilde{g}_m : R \arw R_m$ 
(in particular, all $R_m$ are isomorphic for $m \geq d$).
We use the same construction to define $R_m$ and
$\tilde{g}_m : R \dashrightarrow R_m$ in Theorem \ref{thm_moduli} for $m \geq \lceil d/s \rceil$.
In this sense,
the most essential part of this paper is already written in \cite{St}.
\end{rem}

As stated in Remark \ref{rem_same},
$R_m$ is constructed as a subscheme of the direct product of two Grassmannians.
Hence we have two projections from $R_m$ to the Grassmannians.
By investigating the projections,
we can determine the nef cone of $R_m$ as in the following theorems.
In particular,
this gives an alternative proof of the descriptions of $\Eff(R)$, $\Mov(R)$, and the stable base locus decomposition of $\Eff(R)$
in \cite{Jo}, \cite{Ve}.
In the following theorems,
$\alpha,\beta$ are the base point free line bundles on $R$ defined by Str\o mme 
(see Section \ref{sec_mov} for the definitions of $\alpha,\beta$).

\vspace{2mm}
First,
we describe the cases $r=0$ or $1$. 

\begin{thm}\label{main thm1} 
Assume $r \leq n-2$ and $d \geq 1$,
and set $s= n-r$.
If $r=0$ or $1$,
the following hold.
\begin{enumerate}
\item For each $\lfloor d/s \rfloor +1 \leq m \leq d-1$,
$\tilde{g}_m : R \dashrightarrow R_{m}$ is an SQM of $R$ and
\[
\Nef(R_{m}) =  \r+ ( - (d-m -1)\alpha + \beta) +\r+ ( - (d-m )\alpha + \beta)
\]
holds,
where we regard $\Nef(R_m)$ as a subset of $ N^1(R)_{\R}$ by $\tilde{g}_m^* : N^1(R_m)_{\R} \stackrel{\sim}{\arw} N^1(R)_{\R}$.
Both $- (d-m -1)\alpha + \beta$ and $- (d-m )\alpha + \beta$ are base point free on $R_m$.
\item It holds that
\begin{align*}
\Mov(R) &= \Nef(R) \cup \bigcup_{m=\lfloor d/s \rfloor +1}^{d-1} \Nef(R_m)  \\
&=  \r+ \alpha + \r+ ( - (d- \lfloor d/s \rfloor  -1) \alpha + \beta) , \\
\Eff(R) 
&= \r+ (2d \alpha- r \beta) + \r+ (- (d- \lceil d/s \rceil ) \alpha +\beta ) 
\end{align*}
as in Figure \ref{figure1}.
\item The base point free line bundle $\alpha$ on $R$ defines a surjective morphism
\begin{align*}
f : R &\twoheadrightarrow \P \Big(\bigwedge^r V \otimes H^0(\calo_{\P^1}(d) )^{\vee} \Big),
\end{align*}
which is a fiber type contraction
(resp.\ a divisorial contraction) for $r=0$ (resp.\ $r=1$).
\item The base point free class $- (d- \lfloor d/s \rfloor  -1)\alpha + \beta$ on $R_{\lfloor d/s \rfloor +1}$ defines a surjective morphism to a Grassmannian
\begin{align*}
 R_{\lfloor d/s \rfloor +1}  &\twoheadrightarrow
 \Gr \Bigl( (\lfloor d/s \rfloor +1)s - d , V \otimes H^0 \bigl(\calo_{\P^1}(\lfloor d/s \rfloor )\bigr) \Bigr),
\end{align*}
which is a fiber type contraction (resp.\ a divisorial contraction) if $d/s \not \in \N $ (resp.\ $d/s \in \N $).
\item For each $\lfloor d/s \rfloor +1 \leq m \leq d-1$, the class $  - (d-m -1)\alpha + \beta $, which is base point free on $R_m$ and $R_{m+1}$
(for $m=d-1$, we identify $R_{m+1}=R_d$ with $R$ by the isomorphism $\tilde{g}_d$),
defines birational morphisms $\pr_1, \pr_2$ as in the diagram
\[
\xymatrix{ 
R_{m} \ar@{-->}[rr] \ar@{->>}[rd]_{\pr_2} & &  R_{m+1} \ar@{->>}[ld]^{\pr_1} \\
  & X_m^0 \makebox[0pt]{\hspace{1mm} ,} & \\
 }
\]
where $X_m^0$ is a degeneracy locus in a Grassmannian (see Definition \ref{def_strt_G_m-1} for the definition of $X_m^0$).
Both $\pr_1 $ and $ \pr_2$ in the diagram are small contractions.
\end{enumerate}
\end{thm}

\begin{figure}[htbp]
 \begin{center}

\[
\begin{xy}
(0,0)="O",
(-45,15)="A",
(-37,32)="B",
(-25,36)="C",
(10,42)="D",
(28,36)="E",
(39,25)="F",
(45,15)="G",
(-30,15)*{{\scriptstyle \Nef(R_{\lfloor d/s \rfloor +1})}},
(35,16)*{ {\scriptstyle \Nef(R)}},
(27.8,24.3)*{ {\scriptstyle \Nef(R_{d-1})}},
(15.5,31)*{{\scriptstyle \Nef(R_{d-2})}},
(48,16)*{{\scriptstyle \alpha}},
(41.5,27.5)*{{\scriptstyle \beta}},
(29.5,39.5)*{{\scriptstyle -\alpha+\beta}},
(10,45.5)*{{\scriptstyle -2\alpha+\beta}},
(-47,35)*{ {\scriptstyle - (d- \lfloor d/s \rfloor  -2) \alpha + \beta}},
(-56,18)*{ {\scriptstyle - (d- \lfloor d/s \rfloor  -1) \alpha + \beta}},
(60,15.5)*{},
(-22,25);(4,32)

\ar "O";"A"
\ar "O";"B"
\ar "O";"D"
\ar "O";"E"
\ar "O";"F"
\ar "O";"G"

\crv{~*{.}(-10,33.5)}
\end{xy}
\]

 \end{center}
 \caption{Decomposition of $\Mov(R)$ for $r=0,1$}
 \label{figure1}
\end{figure}
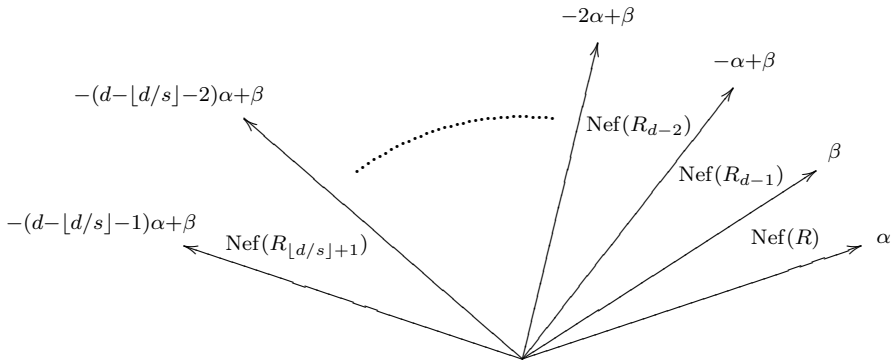

\vspace{4mm}

Next, we consider the cases $2 \leq r \leq n-2$.
Let $R'$ be the Quot scheme parametrizing all rank $n-r$, degree $d$ quotient sheaves of the dual bundle $V^{\vee}_{\P^1}$.
As we wil see in Section \ref{sec_mov},
there exists a natural birational map $R \dashrightarrow R'$.
Applying Theorem \ref{thm_moduli} to $R'$,
we obtain a moduli space $R'_{m'}$ and a birational map $\tilde{g}'_m : R' \dashrightarrow R'_{m'}$ for each $m' \geq \lceil d/r \rceil$.

\begin{thm}\label{main thm2}
Assume $2 \leq r \leq n-2$ and $d \geq 1$,
and set $s= n-r$.
Then
\begin{enumerate}
\item For $\lfloor d/s \rfloor +1 \leq m \leq d-1$ and $\lfloor d/r \rfloor +1 \leq m' \leq d-1$,
$R'$, $R_{m}$,
and $R'_{m'}$ are SQMs of $R$, and it holds that
\begin{align*}
\Nef(R_{m}) 
 &= \r+(- (d-m -1)\alpha + \beta ) + \r+ (-(d-m)\alpha + \beta )  , \\
\Nef(R') &= \r+ ( 2d \alpha -\beta) + \r+ \alpha , \\ 
\Nef(R'_{m'})  
&=  \r+ ( (d + m' + 1)\alpha - \beta) + \r+ ( ( d + m' )\alpha- \beta).
\end{align*}
The classes $- (d-m -1)\alpha + \beta , - (d-m )\alpha + \beta$ are base point free on $R_m$,
$2d \alpha-\beta, \alpha$ are base point free on $R'$,
and $ (d + m' + 1)\alpha - \beta , ( d + m' )\alpha- \beta$ are base point free on $R'_{m'}$.
\item It holds that
\begin{align*}
\Mov(X) &= \Nef(R) \cup \Nef(R') \cup \bigcup_{m=\lfloor d/s \rfloor +1}^{d-1} \Nef(R_m)
 \cup  \bigcup_{m'=\lfloor d/r \rfloor +1}^{d-1} \Nef(R'_{m'}) \\
&= \r+ (  (d + \lfloor d/r \rfloor +1)\alpha - \beta ) + \r+ ( - (d- \lfloor d/s \rfloor  -1) \alpha +\beta ) , \\
\Eff(X) 
&= \r+ ( (d + \lceil d/r \rceil )\alpha - \beta) + \r+ ( - (d- \lceil d/s \rceil  ) \alpha + \beta ) 
\end{align*}
as in Figure \ref{figure2}.
\item The same statement as (4) in Theorem \ref{main thm1} holds.
\item The base point free class $ (d + \lfloor d/r \rfloor +1)\alpha- \beta $ on $R'_{\lfloor d/r \rfloor +1}$
defines a surjective morphism to a Grassmannian
\begin{align*}
 R'_{\lfloor d/r \rfloor +1}  &\twoheadrightarrow
 \Gr \Bigl( (\lfloor d/r \rfloor +1)r - d , V \otimes H^0 \bigl(\calo_{\P^1}(\lfloor d/r \rfloor )\bigr) \Bigr),
\end{align*}
which is a fiber type contraction (resp.\ a divisorial contraction) if $d/r \not \in \N $ (resp.\ $d/r \in \N $).
\item The same statement as (5) in Theorem \ref{main thm1} holds.
\item For each $\lfloor d/r \rfloor +1 \leq m' \leq d-1$, the class $  (d + m' + 1)\alpha - \beta $, which is base point free on $R'_{m'}$ and $R'_{m'+1}$
(for $m'=d-1$, we identify $R_{m'+1}=R'_d$ with $R'$),
defines birational morphisms $\pr_1, \pr_2$ as in the diagram
\[
\xymatrix{ 
R'_{m} \ar@{-->}[rr] \ar@{->>}[rd]_{\pr'_2} & &  R'_{m'+1} \ar@{->>}[ld]^{\pr'_1} \\
  & {X'_m}^{\hspace{-1.5mm} 0} \makebox[0pt]{\hspace{1mm} ,} & \\
 }
\]
where ${X'_m}^{\hspace{-1.5mm} 0}$ is a degeneracy locus in a Grassmannian.
Both $\pr'_1 $ and $ \pr'_2$ in the diagram are small contractions.
\item The class $\alpha$, which is base point free on $R$ and $R'$,
defines birational morphisms $f,f'$ as in the diagram
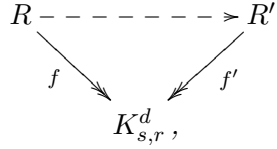
\[
\xymatrix{ 
R \ar@{-->}[rr] \ar@{->>}[rd]_(.4){f} & &  R' \ar@{->>}[ld]^(.4){f'} \\
  & K^d_{s,r} \makebox[0pt]{\hspace{1mm} ,} & \\
 }
\]
where $K^d_{s,r}$, which is defined to be the image of $f$,
is called a quantum Grassmannian.
Both $f$ and $f'$ in the diagram are small contractions.
\end{enumerate}
\end{thm}

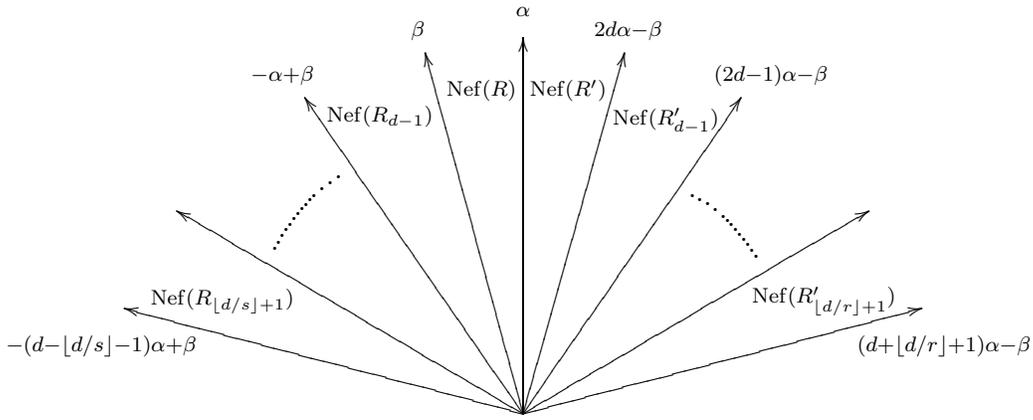
\begin{figure}[htbp]
 \begin{center}

\[
\begin{xy}
(0,0)="O",
(53,14)="A",
(46,27)="B",
(29,42)="C",
(13.5,48)="D",
(0,50)="E",
(-13,48)="F",
(-29,42)="G",
(-46,27)="H",
(-53,14)="I",
(-40,15)*{{\scriptstyle \Nef(R_{\lfloor d/s \rfloor +1})}},
(40,15)*{ {\scriptstyle \Nef(R'_{\lfloor d/r \rfloor +1})}},
(6,43)*{{\scriptstyle \Nef(R')}},
(19,39)*{{\scriptstyle \Nef(R'_{d-1})}},
(-5.5,43)*{{\scriptstyle \Nef(R)}},
(-19,39.5)*{{\scriptstyle \Nef(R_{d-1})}},
(-56,9)*{ {\scriptstyle - (d- \lfloor d/s \rfloor  -1) \alpha + \beta}},
(56,9)*{ {\scriptstyle (d + \lfloor d/r \rfloor +1)\alpha- \beta}},
(0,53.5)*{{\scriptstyle \alpha}},
(-14,51)*{{\scriptstyle  \beta}},
(14,51)*{{\scriptstyle 2d \alpha -\beta}},
(-32,45)*{{\scriptstyle -\alpha+ \beta}},
(33,45)*{{\scriptstyle (2d -1)\alpha -\beta}},
(-33,22);(-24.5,31.5)
**\crv{~*{.}(-29,29)},
(31,21);(22.5,29)
**\crv{~*{.}(26.5,27.5)}

\ar "O";"A"
\ar "O";"B"
\ar "O";"C"
\ar "O";"D"
\ar "O";"E"
\ar "O";"F"
\ar "O";"G"
\ar "O";"H"
\ar "O";"I"

\end{xy}
\]

 \end{center}
 \caption{Decomposition of $\Mov(R)$ for $2 \leq r \leq n-2$}
 \label{figure2}
\end{figure}

\vspace{4mm}
Mori dream spaces, which were introduced by Hu and Keel in \cite{HK},
are special varieties which have very nice properties in view of the minimal model program.
A normal projective variety $X$ is called a Mori dream space
if 
\begin{itemize}
\item[1)] $X$ is $\Q$-factorial and $\Pic(X) \otimes \Q =N^1(X)_{\Q}$,
\item[2)] $\Nef(X )$ is spanned by finitely many semiample line bundles,
\item[3)]  there exists a finite corrections of SQMs $ f_i : X \dashrightarrow X_i$ such that
each $X_i$ satisfies 2) and $\Mov(X) = \bigcup_i f_i^*(\Nef(X_i))$.
\end{itemize}
By Theorems \ref{main thm1}, \ref{main thm2},
we obtain the following corollary.

\begin{cor}\label{cor_MDS}
For $0 \leq r \leq n-1$ and $d \geq 0$,
the Quot scheme $R$ is a Mori dream space.
Moreover,
$R$ is log Fano,
i.e.,
there exists an effective $\Q$-divisor $D$ on $R$ such that
$(R,D)$ is Kawamata log terminal and $-(K_R+D)$ is ample.
\end{cor}

\begin{rem}\label{rem_BCHM}
If the characteristic of the base field $k$ is zero,
any $\Q$-factorial log Fano variety is a Mori dream space by \cite[Corollary 1.3.2]{BCHM}.
\end{rem}

\vspace{2mm}
This paper is organized as follows. In Section \ref{sec_def_Rm},
we recall the definition of $R_m$ in \cite{St} 
and show that $R_m$ parametrizes $[\cale \arw V_{\P^1}]$ satisfying $(\star_m)$.
In Section \ref{sec_sm}, we prove the rest part of Theorem \ref{thm_moduli}. 
In Section \ref{sec_stratification}, we investigate morphisms from $R_m$ to Grassmannians.
In Section \ref{sec_mov}, we give the proof of Theorems \ref{main thm1}, \ref{main thm2}.
Throughout this paper,
we work over a fixed algebraically closed field $k$ of any characteristic.

\subsection*{Acknowledgments}
The author would like to express his gratitude to Professors Yoshinori Gongyo, Yasunari Nagai, and Yujiro Kawamata
for valuable comments and advice.
The author is supported by JSPS KAKENHI Grant Number 261881.

\section{Construction of $R_m$}\label{sec_def_Rm}

Let $V$ be an $n$-dimensional $k$-vector space.
We denote $V \otimes \calo_{T} $ by $V_T$ for a $k$-scheme $T$.
For fixed integers $0 \leq r \leq n-1$ and $d \geq 0$,
let $R$ be the Quot scheme parametrizing all rank $r$, degree $d$ quotient sheaves of $V_{\P^1}$ as in Introduction.
Set $s=n-r$.
We denote by $\Gr(s,V)$ the Grassmannian of $s$-dimensional subspaces of $V$.

For $m \geq 0$,
we set $V_m = V \otimes_{k} H^0 \bigl(\calo_{\P^1}(m) \bigr)$.
Let $G_m =\Gr \bigl( (m+1)s - d , V_m \bigr) $
for $m \in \Z$ such that $(m+1)s - d \geq 0$,
that is,
for $m \geq \lceil d/s \rceil -1$.

As stated in Introduction,
Str\o mme constructed a subscheme $R_m \subset G_{m-1} \times G_m$ and an isomorphism $R \arw R_m$ for each $m \geq d$ in \cite[Section 4]{St}.
By the same construction,
we can define $R_m$ for $ m \geq \lceil d/s \rceil$.
Hence, recall his definition.

\vspace{3mm}
Let
\begin{equation}\label{eq_univ_on R}
0 \arw \cala \arw V_{\P^1 \times R} \arw \calb \arw 0
\end{equation}
be the universal exact sequence on $\P^1 \times R$.
By tensoring $ \calo_{\P^1}(m)$ on (\ref{eq_univ_on R})
and pushing forward by $\pi$,
we obtain a (not necessarily exact) sequence on $R$
\begin{align}\label{eq_*_m}
\tag{$*_m$}
0 \arw \mathscr{A}_m \arw (V_m)_R \arw \mathscr{B}_m \arw 0
\end{align}
for each $m $,
where $\mathscr{A}_m = \pi_* \bigl(\cala(m) \bigr), \mathscr{B}_m=\pi_* \bigl(\calb(m)\bigr)$.
We note that $\mathscr{B}_m$ is locally free of rank $(m+1)r+d$ for $m \geq -1$.
By the universal property of Grassmannians,
we have rational maps as follows.

\begin{defn}[{\cite[Section 4]{St}}]\label{def_g_m}
For $m \geq \lceil d/s \rceil -1$,
we denote by $g_m$ the rational map
$R \dashrightarrow G_m$ induced by the sequence $(*_m)$.
\end{defn}

\begin{rem}\label{rem_g_m}
If $H^1(\cala (m)|_{\P^1 \times \{z\}})=0$ for a point $z \in R$,
$(*_m)$ is exact at $z$.
Hence the rational map $g_m$ is defined at $z$
and $g_m(z)$ is the point in $G_m=\Gr \big((m+1)s-d ,V_m \big)$
corresponding to the subspace $H^0 \big( \cala (m)|_{\P^1 \times \{z\}}  \big) \subset V_m$.
Hence $g_m$ is a morphism for $m \geq d-1$ as stated in {\cite[Section 4]{St}}.
Moreover, $g_m$ is defined on a nonempty open subset of $R$ for each $m \geq \lceil d/s \rceil -1$.
To see this, we take an injection
\[
\iota :  \cale = \calo_{\P^1}( - \lceil d/s \rceil  )^{\oplus s-l} \oplus \calo_{\P^1}( - \lceil d/s \rceil  +1)^{\oplus l} \arw V_{\P^1}
\]
for $l=  \lceil d/s \rceil  s -d \geq 0$.
Then $g_m $ is defined around $z=[V_{\P^1} \arw \coker \iota] \in R$ since $\cala |_{\P^1 \times \{z\}} \cong \cale$
and $H^1(\cale (m))=0$ for $m \geq \lceil d/s \rceil -1$.
\end{rem}

\begin{thm}[{\cite[Theorem 4.2]{St}}]\label{thm_embedding_by_St}
For $m \geq d$,
the morphism $g_m : R \arw G_m$ is a closed embedding.
\end{thm}

Checking the proof of Theorem \ref{thm_embedding_by_St},
we can obtain the following lemma.

\begin{lem}\label{rem_bir}
For $m \geq \lceil d/s \rceil$,
the rational map $g_m : R \dashrightarrow g_m(R) \subset G_m$ is birational.
\end{lem}

\begin{proof}
By the argument in the proof of Theorem 4.2 in \cite{St},
the birationality of $g_m$ follows if $\cala |_{\P^1 \times \{z\}}(m)$ is globally generated for general $z \in R$.
Since $\cala |_{\P^1 \times \{z\}}(m)$ is globally generated if and only if $H^1(\cala |_{\P^1 \times \{z\}}(m-1))=0$,
it is enough to show that $\cala |_{\P^1 \times \{z\}}(m)$ is globally generated for {\it some} $z \in R$
by the upper semicontinuity of cohomology.
For $z=[V_{\P^1} \arw \coker \iota] \in R$ in Remark \ref{rem_g_m},
$\cala |_{\P^1 \times \{z\}}(m) \cong \cale(m)$ is globally generated for $m \geq \lceil d/s \rceil$.
\end{proof}

\begin{rem}\label{rem_not_bir}
Contrary to the case $m \geq \lceil d/s \rceil$,
the rational map $g_{\lceil d/s \rceil -1}$ is not birational for $(r,d) \neq (0,0)$ as follows.

Let $l=  \lceil d/s \rceil  s -d $.
Since $G_{\lceil d/s \rceil -1} =\Gr(l, V_{\lceil d/s \rceil -1})$,
it holds that $\dim G_{\lceil d/s \rceil -1} =l(n\lceil d/s \rceil -l )$ .
By $\dim R =nd+rs$, $n=r+s$, and $l=  \lceil d/s \rceil  s -d$,
we can check that
\[
\dim R_m - \dim G_{\lceil d/s \rceil -1}  =(s-l) \bigl((\lceil d/s \rceil +1)r +d\bigr) ,
\]
which is positive since $s-l >0$ and $(r,d) \neq (0,0)$.
See also Proposition \ref{stratification_of_pr_1_fiber}.
\end{rem}

The following proposition is useful to study $R$.

\begin{prop}[{\cite[Proposition 1.1]{St}}]\label{prop_by_St}
Let $T$ be a locally noetherian $k$-scheme and let $\calf$ be a coherent sheaf on $\P^1 \times T$, flat over $T$.
Assume $R^1 \pi_* \bigl(\calf(-1) \bigr) =0$.
Then there is a short exact sequence on $\P^1 \times T$
\[
0 \arw \pi^* \bigl(\pi_* \calf(-1) \bigr)(-1) \arw \pi^* \pi_* \calf \arw \calf \arw 0.
\]
This sequence is functorial in $\calf$,
and its formulation commutes with arbitrary base change $T' \arw T$.
Its first two terms are locally free.
\end{prop}

By Proposition \ref{prop_by_St},
there is an exact sequence
\[
0 \arw V_{m-1} \otimes \calo_{\P^1}(-1) \arw V_m \otimes \calo_{\P^1} \arw V \otimes \calo_{\P^1}(m) \arw 0
\]
on $\P^1$ for each $m \geq 0$,
which induces an injection $j_m : V_{m-1} \arw V_m \otimes H^0 \bigl(\calo_{\P^1}(1) \bigr)$.

We denote the universal sequence on $G_m$ for $m \geq \lceil d/s \rceil -1$ by
\begin{align}\label{eq_univ_on_G_m}
\tag{${**}_m$} 
0 \arw \mathscr{K}_m \rightarrow (V_m)_{G_m} \rightarrow \mathscr{Q}_m \arw 0.
\end{align}

\vspace{2mm}
Fix $m \geq  \lceil d/s \rceil $.
Pulling back the sequences (${**}_{m-1}$) and (${**}_m$) to $G_{m-1} \times G_m$,
and tensoring the latter by $H:=H^0(\calo_{\P^1}(1)) $,
we obtain a diagram of locally free sheaves on $G_{m-1} \times G_m $:
\begin{equation}\label{eq_def_of_Rm}
\begin{gathered}
\xymatrix{
0 \ar[r] &\pr_1^* \mathscr{K}_{m-1}  \ar[r]^(0.4){i_{m-1}}  & (V_{m-1})_{G_{m-1} \times G_m } \ar[r] \ar[d]^{j_m} & \pr_1^*\mathscr{Q}_{m-1} \ar[r] & 0 \\
0 \ar[r] & \pr_2^*  \mathscr{K}_m \otimes H \ar[r]  & (V_{m})_{G_{m-1} \times G_m } \otimes H \ar[r]^(0.57){p_m}  &  \pr_2^* \mathscr{Q}_m \otimes H \ar[r] & 0  \makebox[0pt]{\hspace{1mm},}
\\
}
\end{gathered}
\end{equation}
where $\pr_1, \pr_2$ are the projections from $G_{m-1} \times G_m$ to $G_{m-1}, G_m$ respectively.
Now we can define $R_m$ as follows.

\begin{defn}[{\cite[Section 4]{St}}]\label{def_R_m}
For $m \geq \lceil d/s \rceil$,
we denote by $R_m \subset G_{m-1} \times G_m $
the closed subscheme defined by the vanishing of $p_m \circ j_m \circ i_{m-1}$.
\end{defn}

\vspace{2mm}
By definition,
there is the following commutative diagram on $R_m$
\begin{equation}
\begin{gathered}
\xymatrix{
0 \ar[r] &(\mathscr{K}_{m-1})_{R_m } \ar[r] \ar[d] & (V_{m-1})_{R_m} \ar[r] \ar[d]^{j_m} & (\mathscr{Q}_{m-1})_{R_m} \ar[r] \ar[d] & 0 \\
0 \ar[r] &(\mathscr{K}_m)_{R_m } \otimes H \ar[r]  & (V_{m})_{R_m } \otimes H \ar[r]  & (\mathscr{Q}_m)_{R_m } \otimes H \ar[r] & 0 \makebox[0pt]{\hspace{1mm},}\\
}
\end{gathered}
\end{equation}
where $(\mathscr{K}_{m-1})_{R_m } $, $(\mathscr{K}_m)_{R_m }$, $(\mathscr{Q}_{m-1})_{R_m}$, $(\mathscr{Q}_m)_{R_m }$ are
the restrictions of $\pr_1^* \mathscr{K}_{m-1}$, $\pr_2^*  \mathscr{K}_m$, $\pr_1^*\mathscr{Q}_{m-1}$, $\pr_2^* \mathscr{Q}_m$ on $R_m$ respectively.
Pulling back this diagram to $\P^1 \times R_m$ by $\pi=\pi_{R_m}$,
tensoring the lower sequence with the natural morphism $H \otimes \calo_{\P^1} \arw \calo_{\P^1}(1)$,
and tensoring the whole diagram by $\calo_{\P^1}(-1)$,
we obtain the following diagram on $\P^1 \times R_m$
\begin{equation}\label{diagram_K_Q}
\begin{gathered}
\xymatrix@C-=0.5cm{
           & 0 \ar[d]                                       & 0 \ar[d]                                             &                                                    & \\
0 \ar[r] & \pi^* (\mathscr{K}_{m-1})_{R_m }(-1) \ar[r] \ar[d]^a     & (V_{m-1})_{\P^1 \times R_m}(-1) \ar[r] \ar[d]^b & \pi^* (\mathscr{Q}_{m-1})_{R_m}(-1) \ar[r] \ar[d]     & 0 \\
0 \ar[r] & \pi^* (\mathscr{K}_m)_{R_m }  \ar[r] \ar[d] & (V_{m})_{\P^1 \times R_m }  \ar[r] \ar[d]     & \pi^*(\mathscr{Q}_m)_{R_m }  \ar[r] \ar[d] & 0 \\
           & \calk(m)  \ar[r]^(0.45)c \ar[d]                             & V_{\P^1 \times R_m} (m) \ar[r] \ar[d]      & \calq (m) \ar[r] \ar[d] & 0 \\
           & 0                                                           &0                                                              &0     \makebox[0pt]{\hspace{1mm},}                                              &   \\
}
\end{gathered}
\end{equation}
where $\calk$ and $\calq$ are defined by the cokernels of the vertical maps.

By tensoring $c$ in the diagram (\ref{diagram_K_Q}) with $\calo_{\P^1}(-m)$,
we have a morphism $\calk \arw V_{\P^1 \times R_m}$.
In the proof of Theorem 4.1 in \cite{St},
Str\o mme showed that $\calk \arw V_{\P^1 \times R_m}$ is injective after restricting to each fiber of $\pi : \P^1 \times R_m \arw R_m$ if $m \geq d$.
For $m \geq \lceil d/s \rceil$,
we have the following lemma.

\begin{lem}\label{lem_K}
For $m \geq \lceil d/s \rceil$,
the morphism $ \calk \arw V_{\P^1 \times R_m}$ satisfies condition $(\star_m)$.
\end{lem}

\begin{proof}
First note that $a$ in the diagram (\ref{diagram_K_Q}) is pointwise injective since so is $b$.
Hence $\calk$ is locally free on $R_m$.
For $z \in R_m$,
we have an exact sequence
\[
0 \arw \big((\mathscr{K}_{m-1})_{R_m} \otimes k(z)\big) \otimes \calo_{\P^1}(-1) \arw \big((\mathscr{K}_m)_{R_m} \otimes k(z)\big) \otimes \calo_{\P^1} \arw \calk(m) |_{\P^1 \times \{z\}} \arw 0
\]
by restricting the left column in the diagram (\ref{diagram_K_Q}) on $\P^1 \times \{z\}$.
By this exact sequence,
$H^1 \big( \calk(m-1) |_{\P^1 \times \{z\}} \big)=0$ holds.
Hence ii) in $(\star_m)$ is satisfied. 
Since
\[
\dim \, (\mathscr{K}_{m-1})_{R_m} \otimes k(z) = ms-d, \quad \dim \, (\mathscr{K}_m)_{R_m } \otimes k(z) = (m+1)s-d,
\]
$\calk |_{\P^1 \times \{z\}}$ is of rank $s$ and degree $-d$.
Thus $ \calk \arw V_{\P^1 \times R_m}$ satisfies i) in $(\star_m)$.
Furthermore,
the linear map
\begin{align*} 
H^0 \bigl(\calk(m) |_{\P^1 \times \{z\}} \bigr) \cong \mathscr{K}_m \otimes k(z)   \arw V_m=V \otimes H^0 \bigl(\calo_{\P^1}(m) \bigr)
\end{align*}
induced by $c$ is injective.
Thus  iii) in $(\star_m)$ holds.
\end{proof}

Using the above diagrams,
Str\o mme proved the following theorem.

\begin{thm}[{\cite[Theorem 4.1]{St}}]\label{R=R_m_by_St}
The morphism
\[
(g_{m-1},g_m) : R \arw G_{m-1} \times G_m
\]
is an isomorphism onto $R_m$ for $m \geq d$.
\end{thm}

By a similar argument to the proof of Theorem \ref{R=R_m_by_St},
we can show that $R_m$ are moduli spaces as in Theorem \ref{thm_moduli} for $m \geq \lceil d/s \rceil$ as follows.

\begin{prop}\label{prop}
For $m \geq \lceil d/s \rceil$,
$R_m$ is the fine moduli space parametrizing equivalence classes $ [\cale \arw V_{\P^1 \times T}]$
which satisfy $(\star_m)$ for each locally noetherian $k$-scheme $T$.
\end{prop}

\begin{proof}
For any morphism $\varphi : T \arw R_m$,
we obtain $(\id_{\P^1} \times \varphi)^* \calk \arw V_{\P^1 \times T}$ by pulling back $\calk \arw V_{\P^1 \times R_m}$ in (\ref{diagram_K_Q}).
By Lemma \ref{lem_K},
$(\id_{\P^1} \times \varphi)^* \calk \arw V_{\P^1 \times T}$ satisfies $(\star_m)$
and we obtain an equivalence class $[(\id_{\P^1} \times \varphi)^* \calk \arw V_{\P^1 \times T}]$.

\vspace{2mm}
Conversely,
we can construct a morphism $T \arw R_m$
from $[\iota : \cale \arw V_{\P^1 \times T} ]$ which satisfies $(\star_m)$ as follows.
Let $\mathscr{C}_{m-1}$ and $\mathscr{C}_m$ be the cokernels of 
\begin{align}\label{eq_2_injection}
 \pi_* \big(\cale(m-1) \big)  \arw (V_{m-1})_T , \quad  \pi_* \big(\cale(m) \big)  \arw (V_{m})_T 
\end{align}
induced by $\iota$ respectively.
Since $\iota$ satisfies $(\star_m)$,
the morphisms (\ref{eq_2_injection}) are injective and
$\mathscr{C}_{m-1}$ and $\mathscr{C}_m$ are locally free of rank $mr+d$ and $(m+1)r+d$ respectively.
By the universal properties of Grassmannians, we obtain a morphism $ \psi: T \arw G_{m-1} \times G_m$,
which satisfies $ \psi^* ( \pr_1^* \mathscr{Q}_{m-1})= \mathscr{C}_{m-1}$ and $\psi^*(\pr_2^* \mathscr{Q}_m)=\mathscr{C}_m$.

Applying Proposition \ref{prop_by_St} to $\cale(m) \arw V_{\P^1 \times T} (m) $,
we obtain
\begin{equation*}
\xymatrix{ 
           & 0 \ar[d]                                       & 0 \ar[d]                                             &                                                    & \\
0 \ar[r] & \pi^* \big(\pi_* \cale(m-1) \big)(-1) \ar[r] \ar[d]     & (V_{m-1})_{\P^1 \times T}(-1) \ar[r] \ar[d] & \pi^* \mathscr{C}_{m-1}(-1) \ar[r] \ar[d]     & 0 \\
0 \ar[r] & \pi^* \big(\pi_* \cale(m) \big)  \ar[r] \ar[d] & (V_{m})_{\P^1 \times T }  \ar[r] \ar[d]     & \pi^* \mathscr{C}_m \ar[r]  & 0 \\
           & \cale (m)  \ar[r] \ar[d]                             & V_{\P^1 \times T} (m) \ar[d]      &  &  \\
           & 0                                                           &0    \makebox[0pt]{\hspace{1mm}.}                                                               &                                                 &   \\
}
\end{equation*}
Tensoring $\calo_{\P^1}(1)$ and pushing forward by $\pi$,
we see that
\[
\psi^* (p_m \circ  j_m \circ i_{m-1}) : \psi^* ( \pr_1^* \mathscr{K}_{m-1})= \pi_* (\cale(m-1)) 
\arw \mathscr{C}_m \otimes H =\psi^*(\pr_2^* \mathscr{Q}_m) \otimes H
\]
is zero. 
Hence $\psi$ factors through $R_m$ since $R_m$ is the locus where $p_m \circ  j_m \circ i_{m-1}$ vanishes.
The cokernels $\mathscr{C}_{m-1}$ and $\mathscr{C}_m$ does not depend on
the choice of the representative $\cale \arw V_{\P^1 \times T}$ of the equivalence class $[\cale \arw V_{\P^1 \times T} ]$.
Thus the morphism $\psi : T \arw R_m$ is defined for the equivalence class $[\cale \arw V_{\P^1 \times T} ]$.
\end{proof}

\section{Smoothness and irreducibility of $R_m$}\label{sec_sm}

In this section,
we prove the rest part of Theorem \ref{thm_moduli},

First,
we show that $R_m$ is smooth and irreducible for any $m \geq \lceil d/s \rceil$.
The proof is similar to that of Theorem 2.1 in \cite{St},
where Str\o mme proved the smoothness and irreducibility of $R$.
However, there is a slight difference.
In the proof, Str\o mme used the universal quotient sheaf $\calb$ on $\P^1 \times R$,
but we use not the quotient $\calq$ but $\calk$ in the diagram (\ref{diagram_K_Q}).

\begin{prop}\label{prop_R_m_is_SQM}
For $m \geq \lceil d/s \rceil $,
$R_m$ is smooth and irreducible. 
\end{prop}

\begin{proof}
Let $M_0 := k^{\oplus s+d},M_{-1} :=k^{\oplus d}$ be vector spaces of dimensions $s+d,d$ respectively.
Let $W$ be the vector space 
\[
W = \Hom_{\P^1}(M_0 \otimes \calo_{\P^1}, M_{-1} \otimes \calo_{\P^1}(1)) \times \Hom (M_0,V)
\]
and let $\overline{X} = \Spec( \Sym W^{\vee})$ be the associated affine space. 
On $\P^1 \times \overline{X} $,
we have a tautological diagram
\[
\xymatrix{
(M_0)_{\P^1 \times \overline{X}} \ar[r]^(0.43){\nu} \ar[d]^{\mu} & (M_{-1})_{\P^1 \times \overline{X}} (1) \\
V_{\P^1 \times \overline{X}} .&  \\
}\]
Let $X_m \subset \overline{X}$ be the open subset defined by the following three conditions;
for each $x \in X_m$,
\begin{itemize}
\item[1)] $\nu |_{\P^1 \times \{x\}}$ is surjective,
\item[2)] $H^1\big((\ker \nu)  |_{\P^1 \times \{x\}}  \otimes \calo_{\P^1}(m-1) \big)=0$,
\item[3)] the induced map $H^0\big( (\ker \nu)  |_{\P^1 \times \{x\}}  \otimes \calo_{\P^1}(m)\big) \arw V_m$ is injective.
\end{itemize}

By 1),
$(\ker \nu) |_{\P^1 \times \{x\}}$ is locally free of rank $s$ and degree $-d$ for $x \in X_m$.
Set $\widetilde{\calk} = (\ker \nu) |_{\P^1 \times X_m}$.
By 1) - 3),
\[
\widetilde{\calk} \hookrightarrow (M_0)_{\P^1 \times X_m}  \stackrel{\mu |_{\P^1 \times X_m}}{\longrightarrow} V_{\P^1 \times X_m} 
\]
satisfies condition $(\star_m)$.
By Proposition \ref{prop},
we have a morphism $g: X_m \arw R_m$.

\vspace{2mm}
For $z \in R_m$,
we can write $\calk |_{\P^1 \times \{z\}} \cong \bigoplus_{i=1}^s \calo(-a_i)$ for some $a_i \in \Z$ by Grothendieck.
Since $\calk \arw V_{\P^1 \times R_m}$ satisfies condition $(\star_m)$ by Lemma \ref{lem_K},
$H^0  \big(\calk |_{\P^1 \times \{z\}} (m) \big) \arw H^0 (V_{\P^1} \otimes (m) )$ is injective.
Hence all $a_i $ are nonnegative.
Thus we can apply Proposition \ref{prop_by_St} to $\calk^{\vee}$,
and we have an exact sequence
\begin{align}\label{eq_dual_K}
0 \arw \pi^* \pi_* \big(\calk^{\vee}(-1)\big)(-1) \arw \pi^* \pi_* (\calk^{\vee})  \arw \calk^{\vee} \arw 0
\end{align}
for $\pi = \pi_{R_m}$.
Since $ \calk^{\vee} |_{\P^1 \times \{z\}}$ is locally free of rank $ s$ and degree $d$ for $z \in R_m$,
$\pi_* (\calk^{\vee})$ and $ \pi_* \big(\calk^{\vee}(-1)\big)$ are locally free of ranks $ s+d$ and $d$ respectively.
Let $Y_0 \arw R_m$ (resp.\ $Y_{-1} \arw R_m$) be the principal $\GL(s+d)$-bundle (resp.\ $\GL(d)$-bundle)
associated to $\pi_*(\calk^{\vee})$ (resp.\ $\pi_* \big(\calk^{\vee}(-1) \big)$) (see \cite[Section 2]{St} for principal $\GL$-bundles).

\vspace{2mm}
Pushing forward the dual sequence of
\begin{align}\label{eq_exact_seq_tilde_K}
0 \arw \widetilde{\calk} \arw (M_0)_{\P^1 \times X_m} \stackrel{\nu}{\arw} (M_{-1})_{\P^1 \times X_m}(1) \arw 0
\end{align}
by $\pi_{X_m}$,
we have an isomorphism $(M_0^{\vee})_{X_m} \arw {\pi_{X_m}}_* (\widetilde{\calk}^{\vee})$.
By this isomorphism,
we can identify the dual of (\ref{eq_exact_seq_tilde_K}) with the exact sequence 
\[
0 \arw \pi_{X_m}^*{\pi_{X_m}}_* \big(\widetilde{\calk}^{\vee}(-1)\big)(-1) \arw \pi_{X_m}^*{\pi_{X_m}}_* (\widetilde{\calk}^{\vee}) \arw \widetilde{\calk}^{\vee} \arw 0
\]
obtained by applying Proposition \ref{prop_by_St} to $\widetilde{\calk}^{\vee}$.
In particular, we have $(M_{-1}^{\vee})_{X_m} \cong {\pi_{X_m}}_* \big( \widetilde{\calk}^{\vee} (-1) \big) $.
By the construction of $g : X_m \arw R_m$,
it holds that $(\id_{\P^1} \times g)^* \calk = \widetilde{\calk}$.
Hence we have
\[
g^*  \pi_* (\calk^{\vee}) = {\pi_{X_m}}_* (\widetilde{\calk}^{\vee}) \cong (M_0^{\vee} )_{X_m} = \calo_{X_m}^{\oplus s+d}.
\]
Similarly,
we have $g^*  \pi_* (\calk^{\vee}(-1)) \cong (M_{-1}^{\vee} )_{X_m} = \calo_{X_m}^{\oplus d} $.
By the universal property of principal $\GL$-bundles,
$g$ factors through $Y :=Y_0 \times_{R_m} Y_{-1}$ as follows.
\[
\xymatrix{ 
X_m \ar[r]^(0.3)\sigma \ar[rd]_g & Y =Y_0 \times_{R_m} Y_{-1} \ar[d]^{\rho} \\
  & R_m \\
 }
\]
We show that $\sigma$ is an isomorphism. 
Since $Y_0$ and $Y_{-1}$ are principal $\GL$-bundles,
we have isomorphisms on $Y$
\begin{align}\label{eq_2isom}
 \rho^*  \pi_* (\calk^{\vee}) \stackrel{\sim}{\arw} \calo_{Y}^{\oplus s+d} = (M_0^{\vee} )_Y, \quad
 \rho^*  \pi_* \big(\calk^{\vee}(-1) \big) \stackrel{\sim}{\arw}  \calo_{Y}^{\oplus d} = (M_{-1}^{\vee} )_Y.
\end{align}
Pulling back (\ref{eq_dual_K}) by $\id_{\P^1} \times \rho$,
composing with the isomorphisms in (\ref{eq_2isom}),
and taking the dual,
we obtain a diagram on $\P^1 \times Y$
\[
\xymatrix{
0 \ar[r] &  (\id_{\P^1} \times \rho)^* \calk \ar[r] \ar[rd] &    (M_0)_{\P^1 \times Y} \ar[r] &  (M_{-1})_{\P^1 \times Y} (1) \ar[r] & 0         \\
  &  & V_{\P^1 \times Y}, & &\\
 }
\]
where $(\id_{\P^1} \times \rho)^* \calk \arw V_{\P^1 \times Y}$ is the pullback of the universal morphism $\calk \arw V_{\P^1 \times R_m}$ by $\id_{\P^1} \times \rho$.
Then there exists a unique lifting $(M_0)_{\P^1 \times Y} \arw V_{\P^1 \times Y} $,
which induces a morphism $u : Y \arw X_m$.
By construction, $u$ is the inverse of $\sigma$.

Thus $R_m$ is smooth
since so are $Y \cong X_m $ and
$\rho: Y \arw R_m  $.
The irreducibility of $R_m$ follows from that of $X_m$.
\end{proof}

\begin{proof}[Proof of Theorem \ref{thm_moduli}]
By Propositions \ref{prop}, \ref{prop_R_m_is_SQM},
it remains to construct a birational map $\tilde{g}_m : R \dashrightarrow R_m$ for each $m \geq \lceil d/s \rceil $.

By construction,
$\iota : \cale \arw V_{\P^1}$ in Remark \ref{rem_g_m} satisfies condition ($\star_m$)  for any $m \geq \lceil d/s \rceil $.
Hence $R_m $ is not empty for any $m \geq \lceil d/s \rceil $.

Let $X_m \subset \overline{X}$ be the open subset in the proof of Proposition \ref{prop_R_m_is_SQM}.
Since $g : X_m  \arw R_m$ in the proof of Proposition \ref{prop_R_m_is_SQM} is surjective,
$X_m $ is nonempty for $m \geq \lceil d/s \rceil $.
Hence $X_m \cap X_d$ is also a nonempty open subset of $\overline{X}$.
This means that both $R_m$ and $R_d$ contain the set
\[
\big\{ [\cale \arw V_{\P^1}] \, \big| \, [\cale \arw V_{\P^1}] \text{ satisfies conditions } (\star_m) \text{ and } (\star_d) \big\}
\]
as a nonempty open subset.
Hence there exists a natural birational map $R_d \dashrightarrow R_m$.
Then the composition of $(g_{d-1},g_d) : R \arw R_d $ and $R_d \dashrightarrow R_m$, which we denote by $\tilde{g}_m$,
is birational by Theorem \ref{R=R_m_by_St} and we finish the proof of Theorem \ref{thm_moduli}.
\end{proof}

\begin{rem}\label{rem_tilde_g_m}
By definition, $\tilde{g}_m$ maps a general point $z \in R$ to $\big[ \cala |_{\P^1 \times \{z\}} \arw V_{\P^1}\big] \in R_m$.
Hence $\tilde{g}_m : R \dashrightarrow R_m \subset G_{m-1} \times G_m$ is nothing but $(g_{m-1},g_m)$.
\end{rem}

\section{Projections to Grassmannians}\label{sec_stratification}

In the previous sections,
we consider Grassmannians of {\it subspace}.
In this section,
we consider Grassmannians of {\it quotient spaces}:
For a vector space $E$,
we denote by $\Gr(E, r)$ the Grassmannian of $r$-dimensional  quotient spaces of $E$.
More generally,
for a coherent sheaf $\mathscr{E}$ on a noetherian scheme $S$,
we set a scheme $\Gr(\mathscr{E},r)$ over $S$ by
\[
\Gr(\mathscr{E},r) := \Quot^{r,\calo_S}_{\mathscr{E}/S/S},
\]
which parametrizes locally free quotient sheaves of $\varphi^* \mathscr{E}$ of rank $r$ for each $\varphi : T \arw S$ (see \cite{Gr}, \cite[5.1.5]{Ni}).
In particular,
the fiber of $\Gr(\mathscr{E},r) \arw S$ over $s \in S$ is the Grassmannian $\Gr(\mathscr{E} \otimes k(s), r)$.
If $\mathscr{E}$ is locally free of rank $n$,
we call $\Gr(\mathscr{E},r) \arw S $ a $\Gr(n,r)$-{\it bundle} over $S$.

\vspace{2mm}
In this section,
we study the projections $\pr_1 : R_m \arw G_{m-1}$ and $\pr_2 : R_m \arw G_m$.
Throughout this section,
we assume that $d \geq 1$.

For each $m \geq 0$,
$j_m : V_{m-1}  \arw V_m \otimes H$ induces a linear map
\[
k_m : V_{m-1} \otimes H^{\vee} \arw V_m .
\]
Then we have a morphism
\[
k_m \circ i'_{m-1}
: \mathscr{K}_{m-1} \otimes H^{\vee} \stackrel{i'_{m-1}}{\longrightarrow}  (V_{m-1})_{G_{m-1}} \otimes H^{\vee} \stackrel{k_m}{\longrightarrow} (V_m)_{G_{m-1}} 
\]
of locally free sheaves on $G_{m-1}$ for $m \geq \lceil d/s \rceil $,
where $i'_{m-1}$ is induced from the natural inclusion $\mathscr{K}_{m-1} \hookrightarrow (V_{m-1})_{G_{m-1}}$.

\begin{lem}\label{prop_Grassmann1}
For $m \geq \lceil d/s \rceil $,
the projection $\pr_1 : R_m \arw G_{m-1}$ is isomorphic to
$ \Gr \left(\coker(k_m \circ i'_{m-1} ), (m+1)r+d \right) \arw G_{m-1}$ over $G_{m-1}$.
\end{lem}

\begin{proof}
Set $\G_m= \Gr \left( \coker(k_m \circ i'_{m-1} ), (m+1)r+d \right)$.
By the natural morphism $(V_m)_{G_{m-1}}  \twoheadrightarrow  \coker(k_m \circ i'_{m-1} )$,
we have a closed embedding
\[
\G_m \hookrightarrow  \Gr \big((V_m)_{G_{m-1}} , (m+1)r+d \big) =G_{m-1} \times G_m
\]
over $G_{m-1}$ since $\Gr(V_m, (m+1)r+d) =\Gr((m+1)s-d,V_m)=G_m$.

By Definition \ref{def_R_m},
$R_{m} \subset G_{m-1} \times G_{m}$ is also defined as the closed subscheme of the vanishing of $p'_{m} \circ k_{m} \circ i'_{m-1} $ for
\begin{equation}\label{eq_def_of_Rm_dual}
\begin{gathered}
\xymatrix{
0 \ar[r] &\pr_1^* \mathscr{K}_{m-1} \otimes H^{\vee}  \ar[r]^(0.43){i'_{m-1}}  & (V_{m-1})_{G_{m-1} \times G_{m} }\otimes H^{\vee} \ar[r] \ar[d]^{k_{m}} & \pr_1^* \mathscr{Q}_{m-1} \otimes H^{\vee}\ar[r] & 0 \\
0 \ar[r] & \pr_2^*  \mathscr{K}_{m}  \ar[r]  & (V_{m})_{G_{m-1} \times G_{m} }  \ar[r]^(0.57){p'_{m}}  &  \pr_2^* \mathscr{Q}_{m}  \ar[r] & 0, \\
}
\end{gathered}
\end{equation}
where we use the same notation $i'_{m-1}$ for the pullback of $i'_{m-1}$ on $G_{m-1}$ by $\pr_1$. 
By the construction of the embedding $\G_m \hookrightarrow G_{m-1} \times G_m$,
the restriction of $p'_m$ on $\G_m$ factors through 
\[
(V_{m})_{G_{m-1} \times G_{m} } |_{\G_m} =\pr_1^* (V_m)_{G_{m-1}}  |_{\G_m} \arw \pr_1^*  \coker(k_m \circ i'_{m-1}) |_{\G_m}.
\]
Thus $p'_{m} \circ k_{m} \circ i'_{m-1} $ vanishes on $\G_m$,
that is,
$\G_m$ is a closed subscheme of $R_m$.
 
On the other hand, 
the restriction of $p'_m$ on $R_m$ factors through
\[
(V_{m})_{G_{m-1} \times G_{m} } |_{R_m}  \arw 
\pr_1^*  \coker(k_m \circ i'_{m-1})|_{R_m}
\]
since $p'_{m} \circ k_{m} \circ i'_{m-1} $ vanishes on $R_m$.
Since $\mathscr{Q}_m$ is locally free of rank $(m+1)r+d$,
the induced surjection $\pr_1^*  \coker(k_m \circ i'_{m-1})|_{R_m} \twoheadrightarrow \pr_2^* \mathscr{Q}_m |_{R_m}$
gives a morphism $R_m \arw \G_m$
by the universal property of $\G_m$.

By construction,
$\G_m \hookrightarrow R_m$ and $R_m \arw \G_m$ are inverses of each other.
Hence $R_m$ coincides with $\G_m$ as a closed subscheme of $G_{m-1} \times G_m$,
and this lemma follows.
\end{proof}

On $G_m$,
we have a morphism
\[
p_m \circ j_m : (V_{m-1})_{G_m} \stackrel{j_m}{\longrightarrow} (V_m)_{G_m} \otimes H \stackrel{p_m}{\longrightarrow} \mathscr{Q}_m \otimes H.
\]
Let $j_m^{\vee} \circ p_m^{\vee} : (\mathscr{Q}_m \otimes H)^{\vee} \arw  (V_{m-1})^{\vee}_{G_m}$ be the dual of $p_m \circ j_m$.

\begin{lem}\label{prop_Grassmann2}
For $m \geq \lceil d/s \rceil $,
the projection $\pr_2 : R_m \arw G_{m}$ is isomorphic to
$ \Gr \left(\coker(j_m^{\vee} \circ p_m^{\vee}  ), ms-d \right) \arw G_{m}$ over $G_{m}$.
\end{lem}

\begin{proof}
By the natural morphism $(V_{m-1})^{\vee}_{G_m} \twoheadrightarrow \coker(j_m^{\vee} \circ p_m^{\vee}  )$,
we have an embedding
\[
\Gr \left(\coker(j_m^{\vee} \circ p_m^{\vee})  ,ms-d \right) \hookrightarrow  \Gr \big((V_{m-1})^{\vee}_{G_m} , ms-d \big) =G_{m-1} \times G_m
\]
over $G_{m}$ since $\Gr((V_{m-1})^{\vee},ms-d)= \Gr(ms-d,V_{m-1})=G_{m-1}$.

Since $R_{m} \subset G_{m-1} \times G_{m}$ is also defined as the closed subscheme of the vanishing of $ i^{\vee}_{m-1} \circ j^{\vee}_{m} \circ p^{\vee}_{m}$ for
\[
\xymatrix{
0  &\pr_1^* \mathscr{K}_{m-1}^{\vee}  \ar[l]  & (V_{m-1})^{\vee}_{G_{m-1} \times G_m }  \ar[l]_(0.52){i^{\vee}_{m-1}}& \pr_1^*\mathscr{Q}_{m-1}^{\vee} \ar[l] & 0  \ar[l] \\
0  & (\pr_2^*  \mathscr{K}_m \otimes H )^{\vee}  \ar[l] & (V_{m} \otimes H)^{\vee}_{G_{m-1} \times G_m }\ar[l] \ar[u]_{j^{\vee}_m} &  (\pr_2^* \mathscr{Q}_m \otimes H)^{\vee} \ar[l]_(0.43){p^{\vee}_m}  & 0  \ar[l] \makebox[0pt]{\hspace{1mm},}
\\
}
\]
we can show that $R_m $ coincides with $ \Gr \left(\coker(j_m^{\vee} \circ p_m^{\vee}) ,ms-d \right)$ as a closed subscheme of $G_{m-1} \times G_m$
by an argument similar to that in the proof of Lemma \ref{prop_Grassmann1}.
We leave the detail to the reader.
\end{proof}

\subsection{Stratification of $\pr_2$}

In this subsection,
we consider a stratification of $\pr_2 : R_m \arw G_m$.

\begin{defn}\label{def_strt_G_m-1}
For $m \geq \lceil d/s \rceil $ and $i \in \Z$,
we define a closed subscheme $X_{m}^i $ of $G_{m}$ to be the $mr+d-i$-th degeneracy locus of 
$j_m^{\vee} \circ p_m^{\vee} : (\mathscr{Q}_m \otimes H)^{\vee} \arw  (V_{m-1})^{\vee}_{G_m}$
(see \cite[Chapter 2]{ACGH} for degeneracy loci).
\end{defn}

\begin{lem}\label{lem_X_m^0}
For $m \geq \lceil d/s \rceil $,
$X_m^0$ is reduced, irreducible, Cohen-Macaulay,
and $X_m^0=\pr_2 (R_m)$.
Moreover,
$\pr_2: R_m \arw X_m^0$ is an isomorphism over the non-empty open subset $X_m^0 \setminus X_m^1 $.
In particular,
$X_m^0 \setminus X_m^1 $ is smooth.
\end{lem}

\begin{proof}
By Theorem \ref{thm_moduli}, Lemma \ref{rem_bir},
$\tilde{g}_m : R \dashrightarrow R_m$ and $g_m : R \dashrightarrow g_m(R) \subset G_m$ are birational.
Since $g_m = \pr_2 \circ \tilde{g}_m $ holds by Remark \ref{rem_tilde_g_m},
$\pr_2 : R_m \arw \pr_2(R_m) \subset G_m$ is birational as well.

We show that $\pr_2(R_m) =X_m^0$ holds set-theoretically.
By the definition of $X_m^i$,
$\coker(j_m^{\vee} \circ p_m^{\vee}  )  |_{X_m^i \setminus X_m^{i+1}}$
is local free of rank
\[
\rank (V_{m-1})^{\vee}_{G_m} -(mr+d-i)=ms-d+i.
\]
Hence $\pr_2(R_m)= X_m^0$ and 
$\pr_2 : R_m \arw G_m$ is a $\Gr(ms-d+i,ms-d)$-bundle over $X_m^i \setminus X_m^{i+1}$ for $i \geq 0$
by Lemma \ref{prop_Grassmann2}.
In particular,
$X_m^0$ is irreducible
and $\pr_2: R_m \arw X_m^0$ is an isomorphism over $X_m^0 \setminus X_m^1 $.
Since $\pr_2: R_m \arw X_m^0$ is birational and $\pr_2 : R_m \arw G_m$ is not finite over $X_m^1$,
$X_m^0 \setminus X_m^1 $ is non-empty

Since $X_m^0 \subset G_m$ is the $ mr+d$-th degeneracy locus of 
$(\mathscr{Q}_m \otimes H)^{\vee} \arw  (V_{m-1})^{\vee}_{G_m}$,
the expected codimension of $X_m^0 $ in $G_m$ is
\[
(\rank (V_{m-1})^{\vee}_{G_m} - (mr+d) )(\rank (\mathscr{Q}_m \otimes H)^{\vee} -(mr+d)).
\]
Since $\dim X^0_m =\dim R_m =nd+rs$,
we can check that the expected codimension coincides with $\dim G_m - \dim X_m^0$.
Hence $X_m^0$ is Cohen Macaulay by \cite[Chapter II, (4.1) Proposition]{ACGH}.

Since $\pr_2: R_m \arw X_m^0$ is an isomorphism over $X_m^0 \setminus X_m^1 $,
$X_m^0 \setminus X_m^1$ is smooth.
Since $X_m^0$ is Cohen-Macaulay and reduced on the non-empty open subset $X_m^0 \setminus X_m^1$,
$X_m^0$ is reduced.
\end{proof}

\begin{prop}\label{stratification_pr_2}
For $m \geq \lceil d/s \rceil $, the following hold.
\begin{enumerate}
\item For $0 \leq i \leq \lfloor d/(m+1) \rfloor$, $X_m^i$ is irreducible, Cohen-Macaulay, $X_m^i \setminus X_m^{i+1}$ is smooth, and $\dim X_m^i= n(d-(m+1)i) + (r+i)(s-i)$.
\item $X_m^i$ is normal for $0 \leq i \leq \lfloor d/(m+1) \rfloor$.
\item $X_m^i=\emptyset$ for $i > \lfloor d/(m+1) \rfloor$.
\item $\pr_2 : R_m \arw G_{m}$ is a $\Gr( ms-d+i, ms-d)$-bundle over $X_m^i \setminus X_m^{i+1}$.
\end{enumerate}
In particular,
$\pr_2 : R_m \arw X_m^0 \subset G_m$ is an isomorphism in codimension one. 
\end{prop}

\begin{proof}
We note that (4) is already shown in the second paragraph of the proof of Lemma \ref{lem_X_m^0},

First, we show (1).
In this proof, we denote $R, R_m, X_m^i$ by $R(s,d),R_m(s,d), X_m^i(s,d)$ respectively when we need to clarify $s, d$.
Since
\[
(m+1)s-d = (m+1)(s-i) - (d- (m+1)i),
\]
we have a morphism 
\[
\pr_2^i : R_m(s-i, d-(m+1)i) \arw G_m= \Gr((m+1)s-d, V_m ) 
\]
as $\pr_2 :  R_m \arw G_m$ for each $0 \leq i \leq \lfloor d/(m+1) \rfloor$.

Applying Lemma \ref{lem_X_m^0} to $R_m(s-i,d-(m+1)i)$,
the image of $\pr_2^i$ is $X_m^0(s-i, d-(m+1)i) $,
which is the $m(r+i) +(d-(m+1)i)$-th degeneracy locus of
$j_m^{\vee} \circ p_m^{\vee} : (\mathscr{Q}_m \otimes H)^{\vee} \arw  (V_{m-1})^{\vee}_{G_m}$.
Since $m(r+i) +(d-(m+1)i)=mr+d-i$,
$X_m^0(s-i, d-(m+1)i) $ coincides with $X_m^i(s,d)$.
Similarly,
$X_m^1(s-i, d-(m+1)i) =X_m^{i+1}(s,d)$ holds for each $i$.
Applying Lemma \ref{lem_X_m^0} to 
$X_m^0(s-i,d-(m+1)i)$ and $X_m^1(s-i,d-(m+1)i)$,
we obtain (1).

\vspace{2mm}
We show (3).
It suffices to see that
\[
X_m^{ \lfloor d/(m+1) \rfloor +1}(s,d) = X_m^1 (s- \lfloor d/(m+1) \rfloor, d-(m+1) \lfloor d/(m+1) \rfloor )
\]
is empty.
By (4),
$X_m^1 (s- \lfloor d/(m+1) \rfloor, d-(m+1) \lfloor d/(m+1) \rfloor )$ is the image of the exceptional locus of $\pr_2^{ \lfloor d/(m+1) \rfloor}$.
Hence it suffices to show that $\pr_2^{ \lfloor d/(m+1) \rfloor}$ is an embedding.
Since $d - (m+1) \lfloor d/(m+1) \rfloor \leq m$,
\[
R(s- \lfloor d/(m+1) \rfloor, d-(m+1) \lfloor d/(m+1) \rfloor) =R_m(s- \lfloor d/(m+1) \rfloor, d-(m+1) \lfloor d/(m+1) \rfloor)
\]
holds by Theorem \ref{R=R_m_by_St}.
Thus
$\pr_2^{ \lfloor d/(m+1) \rfloor}$ is an embedding by Theorem \ref{thm_embedding_by_St} and (3) is shown.

\vspace{2mm}
To show (2),
it suffices to see that $X_m^0$ is normal
since $X_m^i(s,d)= X_m^0(s-i, d-(m+1)i)$.

By (1), (3),
$X_m^0 \setminus X_m^1$ is smooth and $\dim X_m^1 \leq \dim X_m^0 -2$.
Hence
$X_m^0$ is smooth in codimension one.
Since $X_m^0$ is Cohen-Macaulay,
$X_m^0$ is normal by Serre's criterion.

\vspace{2mm}
We show the last statement.
If $m \geq d$,
$X_m^1= \emptyset$ by (3) since $1 > \lfloor d/(m+1) \rfloor $.
Thus $\pr_2 : R_m \arw X_m^0$ is an isomorphism by (4).

If $ \lceil d/s \rceil \leq m \leq d-1$,
it holds that $ \lfloor d/(m+1) \rfloor \geq 1$, hence we can compute the dimension of the exceptional locus $\pr_2^{-1}(X_m^1) \subset R_m$ as
\begin{align*}
\dim \pr_2^{-1}(X_m^1) &= \max_{1 \leq i \leq \lfloor d/(m+1) \rfloor}  n(d-(m+1)i)+(r+i)(s-i) +  i(ms-d)  \\
&= nd+rs -(m+2)r-d-1 
\end{align*}
by (1), (3), and (4).
Since $d \geq 1$, $\dim \pr_2^{-1}(X_m^1)  \leq \dim R_m -2$ holds.
Hence $\pr_2 : R_m \arw Y_m^0 \subset G_m$ is an isomorphism in codimension one. 
\end{proof}

\subsection{Stratification of $\pr_1$}
In this subsection, we consider $\pr_1$.
As we will see,
$\pr_1 : R_m \arw G_{m-1}$ is birational for $m \geq \lceil d/s \rceil +1$,
and is not for $m=\lceil d/s \rceil $.
Hence we study $\pr_1 : R_m \arw G_{m-1}$ for $m \geq \lceil d/s \rceil +1$,
namely, $\pr_1 : R_{m+1} \arw G_m$ for $m \geq \lceil d/s \rceil $ first.
We study $\pr_1 : R_{\lceil d/s \rceil } \arw G_{\lceil d/s \rceil -1}$ next.

\begin{prop}\label{stratification_of_pr_1}
For $m \geq \lceil d/s \rceil $, it holds that 
\begin{itemize}
\item[(a)]  $\pr_1(R_{m+1})=X_m^0 \subset G_m$,
\item[(b)] $\pr_1 : R_{m+1} \arw X_m^0 $ is a $\Gr((m+2)r+d +i, (m+2)r+d)$-bundle over $X_m^i \setminus X_m^{i+1}$ for each $0 \leq i \leq \lfloor d/(m+1) \rfloor$.
\end{itemize}
In particular,
$\pr_1 : R_{m+1} \arw X_m^0$ is an isomorphism in codimension one for $m \geq \lfloor d/s \rfloor +1 $.
On the other hand,
$\pr_1 : R_{(d/s) +1} \arw X_{d/s}^0 = G_{d/s}$ contracts a divisor if $d/s \in \N$.
\end{prop}

\begin{proof}
Since $g_{m} $ coincides with $ \pr_1 \circ \tilde{g}_{m+1} : R \dashrightarrow R_{m+1} \arw \pr_1(R_{m+1})  \subset G_m$ and $g_m, \tilde{g}_{m+1}$ are birational,
$\pr_1 : R_{m+1} \arw \pr_1(R_{m+1})  \subset G_{m}$ is birational for $m \geq \lceil d/s \rceil $.

For $i \geq 0$, 
let $Y_m^i  \subset G_m$ be the $(m+2)s-d-i$-th degeneracy locus of
$k_{m+1} \circ i'_{m} : \mathscr{K}_{m} \otimes H^{\vee} \arw (V_{m+1})_{G_{m}} $.
By the same argument as Lemma \ref{lem_X_m^0},
Proposition \ref{stratification_pr_2},
we can show that $\pr_1(R_{m+1})=Y_m^0$ and
\begin{enumerate}
\item[(1)'] For $0 \leq i \leq \lfloor d/(m+1) \rfloor$, $Y_m^i$ is irreducible, Cohen-Macaulay, $Y_m^i \setminus Y_m^{i+1}$ is smooth, and $\dim Y_m^i= n(d-(m+1)i) + (r+i)(s-i)$.
\item[(2)'] $Y_m^i$ is normal for $0 \leq i \leq \lfloor d/(m+1) \rfloor$.
\item[(3)'] $Y_m^i=\emptyset$ for $i > \lfloor d/(m+1) \rfloor$.
\item[(4)'] $\pr_1 : R_{m+1} \arw G_{m}$ is a $\Gr((m+2)r+d +i, (m+2)r+d)$-bundle over $Y_m^i \setminus Y_m^{i+1}$.
\end{enumerate}
Hence it suffices to show $X_m^i= Y_m^i$
for (a), (b).

\vspace{2mm}
Since $g_m = \pr_1 \circ \tilde{g}_{m+1} = \pr_2 \circ \tilde{g}_m : R \dashrightarrow G_m$,
we have $\overline{g_m(R)} = \pr_1(R_{m+1}) = \pr_2(R_m)$,
where $\overline{g_m(R)} $ is the closure of $g_m(R)$.
Since $\pr_1(R_{m+1}) =Y_m^0$ and $\pr_2(R_m) = X_m^0$,
it holds that $X_m^0=Y_m^0$.

For $0 \leq i \leq \lfloor d/(m+1) \rfloor$,
we denote $X_m^i, Y_m^i$ by $X_m^i(s,d), Y_m^i(s,d)$ respectively to clarify $s,d$.
As in the proof of Proposition \ref{stratification_of_pr_1},
it holds that $X_m^i(s,d)=X_m^0(s-i,d-(m+1)i)$.
By a similar argument,
$Y_m^i(s,d)=Y_m^0(s-i,d-(m+1)i)$ holds.
Hence we have $X_m^i(s,d)=X_m^0(s-i,d-(m+1)i)  =Y_m^0(s-i,d-(m+1)i) =Y_m^i(s,d)$.
Thus (a), (b) are proved.

\vspace{2mm}
If $m \geq d$,
$\pr_1 : R_{m+1} \arw Y_m^0$ is an isomorphism by (3)' and (4)'.
For $\lfloor d/s \rfloor +1 \leq m \leq d-1$, 
the dimension of the exceptional locus $\pr_1^{-1}(Y_{m}^1) \subset R_{m+1}$ of $\pr_1$ is
\[
 nd+rs -(sm-d+1) =\dim R_{m+1} -(sm-d+1).
\]
by (1)', (3)', and (4)'.
Hence $\pr_1 : R_{m+1} \arw Y_m^0 $ is an isomorphism in codimension one for $\lfloor d/s \rfloor +1 \leq m \leq d-1$.
On the other hand,
$\pr_1 : R_{(d/s) +1} \arw Y_{d/s}^0$ contracts a divisor if $d/s \in \N$.
In this case,
$Y_{d/s}^0=G_{d/s} $ holds since $\dim Y_{d/s}^0=nd+rs = \dim G_{d/s}$.
\end{proof}


In Definition \ref{def_strt_G_m-1} and the proof of Proposition \ref{stratification_of_pr_1},
we defined $X_m^i =Y_m^i \subset G_m$ for $m \geq \lceil d/s \rceil$.
For $m=\lceil d/s \rceil-1$,
we define $X_{\lceil d/s \rceil-1}^i \subset G_{\lceil d/s \rceil-1}$ in a slightly different manner.

\begin{defn}\label{def_Y^i_2}
Set $l=  \lceil d/s \rceil  s -d$.
For $i \geq 0$,
we define a closed subscheme $X_{\lceil d/s \rceil-1}^i $ of $ G_{\lceil d/s \rceil-1}=\Gr(l, V_{\lceil d/s \rceil-1})$ to be the $2l-i$-th degeneracy locus of 
\[
k_{\lceil d/s \rceil} \circ i'_{\lceil d/s \rceil-1} : \mathscr{K}_{\lceil d/s \rceil-1} \otimes H^{\vee} \arw (V_{\lceil d/s \rceil})_{G_{\lceil d/s \rceil-1}} .
\]
Since $\rank \mathscr{K}_{\lceil d/s \rceil-1} \otimes H^{\vee} =2l$, it holds that $X_{\lceil d/s \rceil-1}^0=G_{\lceil d/s \rceil-1}$.
\end{defn}

\begin{prop}\label{stratification_of_pr_1_fiber}
In the case $m=\lceil d/s \rceil -1$,
the following hold.
\begin{enumerate}
\item For $0 \leq i \leq  l -\big\lceil l/ \lceil d/s\rceil \big\rceil$, $X_{\lceil d/s \rceil-1}^i $ is irreducible, normal, Cohen-Macaulay,
$X_{\lceil d/s \rceil-1}^i  \setminus X_{\lceil d/s \rceil-1}^{i+1} $ is smooth, and $\dim X_{\lceil d/s \rceil-1}^i = n\big( (\lceil d/s\rceil-1)l - \lceil d/s\rceil i \big) + (n-l+i)(l-i)$.
\item $X_{\lceil d/s \rceil-1}^i =\emptyset$ for $i > l -\big\lceil l/ \lceil d/s\rceil \big\rceil$.
\item $\pr_1 : R_{ \lceil d/s \rceil } \arw G_{ \lceil d/s\rceil  -1}$ is a $\Gr \big( (\lceil d/s \rceil +1) n -2l +i,  (\lceil d/s \rceil  +1)r +d \big)$-bundle
over $X_{\lceil d/s \rceil-1}^i  \setminus X_{\lceil d/s \rceil-1}^{i+1}$ for $0 \leq i \leq  l -\big\lceil l/ \lceil d/s\rceil \big\rceil$.
\end{enumerate}
In particular, $\pr_1(R_{ \lceil d/s \rceil })=X_{\lceil d/s \rceil-1}^0 =G_{\lceil d/s\rceil-1}$ holds and
any fiber of $\pr_1 : R_{ \lceil d/s \rceil } \arw G_{ \lceil d/s\rceil  -1}$ is a Grassmannian, which is not a point.
\end{prop}

\begin{proof}
Since 
\[
\lceil d/s \rceil s -d=l=   \lceil d/s \rceil l -  (\lceil d/s \rceil -1)l,
\]
$G_{ \lceil d/s \rceil -1} =\Gr(l, V_{\lceil d/s \rceil-1})$ does not change by replacing $s,d$ with $\bar{s}:=l, \bar{d}:=(\lceil d/s \rceil-1)l$.
Since $m=\lceil d/s \rceil -1 \geq \lceil \bar{d} / \bar{s} \rceil $,
we can define $X_{\lceil d/s \rceil -1}^i(\bar{s},\bar{d}) \subset G_{ \lceil d/s \rceil -1}$
by Definition \ref{def_strt_G_m-1}.
Since $X_{\lceil d/s \rceil -1}^i(\bar{s},\bar{d}) = Y_{\lceil d/s \rceil -1}^i(\bar{s},\bar{d})$
is the $(\lceil d/s \rceil +1)\bar{s}-\bar{d}-i$-th degeneracy locus of
\[
k_{\lceil d/s \rceil } \circ i'_{\lceil d/s \rceil -1} : \mathscr{K}_{\lceil d/s \rceil -1} \otimes H^{\vee} \arw (V_{\lceil d/s \rceil })_{G_{\lceil d/s \rceil -1}} 
\]
and $(\lceil d/s \rceil +1)\bar{s}-\bar{d}-i =2l-i$,
$X_{\lceil d/s \rceil-1}^i$ in Definition \ref{def_Y^i_2} coincides with $X_{\lceil d/s \rceil -1}^i(\bar{s},\bar{d})$.

Applying (1), (2), (3) in Proposition \ref{stratification_pr_2} to $X_{\lceil d/s \rceil -1}^i(\bar{s},\bar{d})$,
we have (1), (2) in this proposition.

By definition,
the restriction of $\coker (k_{\lceil d/s \rceil } \circ i'_{\lceil d/s \rceil -1} )$ on $X_{\lceil d/s \rceil-1}^i  \setminus X_{\lceil d/s \rceil-1}^{i+1}$
is locally free of rank $(\lceil d/s \rceil +1) n -2l +i$.
Hence (3) follows from Lemma \ref{prop_Grassmann1}.

\vspace{2mm}
For the last statement,
it suffices to show that a general fiber of $ \pr_1$ is not a point.
By (3),
the general fiber of $\pr_1 : R_{ \lceil d/s \rceil } \arw G_{ \lceil d/s\rceil  -1}=X_{\lceil d/s \rceil-1}^0$ is the Grassmannian $\Gr \big( (\lceil d/s \rceil +1) n -2l,  (\lceil d/s \rceil  +1)r +d \big)$,
which is not a point since
\[
(\lceil d/s \rceil +1) n -2l -   (\lceil d/s \rceil +1)r -d  =s-l > 0
\]
and  $(\lceil d/s \rceil +1)r+d >0$.
\end{proof}

\section{Movable and effective cones of $R$}\label{sec_mov}

In this section,
we prove Theorems \ref{main thm1}, \ref{main thm2}.
Throughout this section,
we assume $d \geq 1$ and $0 \leq r \leq n-2$.

Str\o mme defined line bundles $\alpha, \beta$ on $R$ by
$\alpha= c_1(\mathscr{B}_d ) -c_1(\mathscr{B}_{d-1})$ and $\beta = c_1(\mathscr{B}_{d-1})$
for $ \mathscr{B}_m=\pi_* \bigl(\calb(m)\bigr)$.
For $m \geq 0$,
there exists an exact sequence $ 0 \arw \mathscr{B}_{m-1} \arw \mathscr{B}_m^{\oplus 2} \arw \mathscr{B}_{m+1} \arw 0$
induced by $0 \arw \calo_{\P^1}(-1) \arw \calo_{\P^1}^{\oplus 2} \arw \calo_{\P^1}(1) \arw 0$ (see \cite[Lemma 5.1]{St}).
Hence it holds that
$c_1(\mathscr{B}_m) =- (d-1-m)\alpha + \beta $ for $m \geq -1$.

\begin{lem}\label{lem_SQM}
For $\lfloor d/s \rfloor +1 \leq m \leq d-1$,
the birational map $\tilde{g}_m : R \dashrightarrow R_{m}$ is an SQM of $R$.
Under the identification of $N^1(R_m)_{\R}$ with $N^1(R)_{\R}$,
it holds that
\begin{align*}
\Nef(R_m) &= \r+ c_1(\mathscr{B}_{m-1}) + \r+ c_1(\mathscr{B}_{m}) \\
&= \r+ ( - (d-m)\alpha + \beta) + \r+ ( - (d-1-m)\alpha +\beta),
\end{align*}
and $c_1(\mathscr{B}_{m-1})$ and $c_1(\mathscr{B}_{m})  $ are base point free on $R_m$.
\end{lem}

\begin{proof}
Consider the following diagram
\begin{equation}\label{diagram_small_contractions}
\begin{gathered}
\xymatrix{ 
R_{m+1} \ar@{-->}[rr] \ar[rd]_{\pr_1} & &  R_m \ar[ld]^{\pr_2} \\
  & X_m^0 \makebox[0pt]{\hspace{13mm}$\subset  G_m$} & \\
 }
\end{gathered}
\end{equation}
for $\lfloor d/s \rfloor +1 \leq m \leq d-1$.
By Propositions \ref{stratification_pr_2}, \ref{stratification_of_pr_1},
the birational morphism $\pr_1,\pr_2$ in the diagram (\ref{diagram_small_contractions}) are isomorphisms in codimension one.
Hence $R_{m+1} \dashrightarrow R_m$ is also an isomorphism in codimension one.
Since $\tilde{g}_m : R \dashrightarrow R_{m}$ is decomposed as
\[
R \xrightarrow[\tilde{g}_d]{\sim} R_d \dashrightarrow R_{d-1} \dashrightarrow \cdots \dashrightarrow R_{m+1} \dashrightarrow R_m,
\]
$\tilde{g}_m : R \dashrightarrow R_{m}$ is an isomorphism in codimension one as well.
Since $R_m$ is smooth by Proposition \ref{prop_R_m_is_SQM},
$\tilde{g}_m : R \dashrightarrow R_{m}$ is an SQM of $R$.

By Propositions \ref{stratification_pr_2}, \ref{stratification_of_pr_1}, and \ref{stratification_of_pr_1_fiber},
$\pr_1 : R_m \arw G_{m-1}$ and $\pr_2 : R_m \arw G_m$ are not finite morphisms for $\lfloor d/s \rfloor +1 \leq m \leq d-1$.
Hence $\Nef(R_m) $ is spanned by $\pr_1^* c_1(\mathscr{Q}_{m-1})$ and $\pr_2^* c_1(\mathscr{Q}_m)$
since $c_1(\mathscr{Q}_{m-1})$ and $c_1(\mathscr{Q}_m)$ are ample line bundles on the Grassmannians $G_{m-1} $ and $G_m$ respectively.
By the definition of $g_m : R \dashrightarrow G_m$,
$c_1(\mathscr{B}_m) $ is the pullback of $c_1(\mathscr{Q}_m)$ by $g_m$
(note that $g_m$ is defined outside codimension two locus since so is $\tilde{g}_m : R \dashrightarrow R_m$).
Hence, $\pr_2^*c_1(\mathscr{Q}_m)  =c_1(\mathscr{B}_m) $ holds under the identification of $N^1(R_m)_{\R}$ with $N^1(R)_{\R}$. 
Similarly,
$\pr_1^* c_1(\mathscr{Q}_{m-1}) =c_1(\mathscr{B}_{m-1}) $ holds.
Hence $\Nef(R_m)$ is spanned by
the two base point free line bundles $c_1(\mathscr{B}_{m-1})$ and $c_1(\mathscr{B}_{m}) $. 
\end{proof}

\begin{rem}\label{rem_small_contraction}
By Lemma \ref{lem_SQM},
the Picard number of $R_m$ is two for $\lfloor d/s \rfloor +1 \leq m \leq d-1$.
Hence $\pr_1$ and $\pr_2$ in the diagram (\ref{diagram_small_contractions}) are small contractions for $\lfloor d/s \rfloor +1 \leq m \leq d-1$
since $X_m^0$ is normal by Proposition \ref{stratification_pr_2}. 
\end{rem}

\begin{lem}\label{lem_edge_Mov_Eff}
The morphism $\pr_1 : R_{\lfloor d/s \rfloor+1}  \arw G_{\lfloor d/s \rfloor}$ is a fiber type contraction (resp.\ a divisorial contraction)
if $d/s \not \in \N$ (resp.\ $d/s \in \N$).
Furthermore,
$\r+ c_1(\mathscr{B}_{\lfloor d/s \rfloor})$ is an edge of $\Mov(R)$,
and $\r+ c_1(\mathscr{B}_{\lceil d/s \rceil-1})$ is an edge of $\Eff(R)$.
\end{lem}

\begin{proof}
First,
assume $d/s \not \in \N$.
Since the Picard number of $R_{\lfloor d/s \rfloor+1} $ is two,
$\pr_1 : R_{\lceil d/s \rceil}  \arw G_{\lceil d/s \rceil-1  }$ is a fiber type contraction by Propositions \ref{stratification_of_pr_1_fiber}.
Thus $\pr_1^* c_1(\mathscr{Q}_{\lceil d/s \rceil-1}) =c_1(\mathscr{B}_{\lceil d/s \rceil-1})$ spans an edge of both $\Mov(R_{\lceil d/s \rceil})$ and $\Eff(R_{\lceil d/s \rceil})$.
Since $\lceil d/s \rceil-1 =\lfloor d/s \rfloor$ and $R_{\lceil d/s \rceil}=R_{\lfloor d/s \rfloor+1}$ is an SQM of $R$,
this lemma holds if $d/s \not \in \N$.

\vspace{2mm}
Next, assume $d/s \in \N$.
Since the Picard number of $R_{\lfloor d/s \rfloor+1} $ is two,
$\pr_1 : R_{(d/s) +1}  \arw G_{d/s  }$ is a divisorial contraction by Propositions \ref{stratification_of_pr_1}.
Hence $\pr_1^* c_1(\mathscr{Q}_{d/s }) =c_1(\mathscr{B}_{d/s})$ spans an edge of $\Mov(R_{(d/s) +1})$
and $E$ spans an edge of $\Eff(R_{(d/s) +1})$, where $E$ is the contracted divisor of $\pr_1 : R_{(d/s) +1}  \arw G_{d/s  }$.
To compute the class of $E$,
we compare the canonical divisors on $R_{ (d/s) +1}$ and $ G_{ d/s } $.
By \cite[Theorem 7.1 (ii)]{St},
it holds that $K_R= - (n+(2r+2-n)d)\alpha-(n-2r)\beta $.
On the other hand,
$K_{ G_{ d/s } } = - n((d/s)+1) c_1(\mathscr{Q}_{d/s}) $ holds since $ G_{d/s }  $ is a Grassmannian.
Hence we have
\begin{align*}
K_{R_{(d/s)+1}} - \pr_1^* K_{ G_{ d/s } } &= - (n+(2r+2-n)d)\alpha  -(n-2r)\beta + n((d/s)+1) c_1(\mathscr{B}_{d/s})\\
&= ((nd/s) +2r ) c_1(\mathscr{B}_{(d/s)-1} ).
\end{align*}
Thus the class of $E$ is a positive multiple of $c_1(\mathscr{B}_{(d/s)-1} ) =c_1(\mathscr{B}_{\lceil d/s \rceil -1} )$
and this class spans an edge of $\Eff(R_{(d/s) +1})=\Eff(R)$.
\end{proof}

To find the other edges of $\Mov(R)$ and $\Eff(R)$,
we recall the morphism defined by $\alpha = c_1(\mathscr{B}_d) -c_1(\mathscr{B}_{d-1})$
(see \cite[Lemma 6.4]{St}).
Taking the $s$-th exterior power of the universal inclusion $\cala \hookrightarrow V_{\P^1 \times R}$,
we have $\det \cala  \arw \bigwedge^{s} V_{\P^1 \times R}$.
By tensoring $ \calo_{\P^1}(d)$ and taking $\pi_*$,
we obtain a nowhere vanishing monomorphism
\[
\calo_{R}(-\alpha) \arw \bigwedge^s V_R \otimes H^0(\calo_{\P^1}(d)).
\]
This induces a morphism
$f : R \arw \P\big(\bigwedge^r V \otimes H^0(\calo_{\P^1}(d))^{\vee}\big)$.

\begin{defn}\label{def_quantum_Gr}
We define $K^d_{s,r} \subset  \P\big(\bigwedge^r V \otimes H^0(\calo_{\P^1}(d))^{\vee}\big)$ to be the image $ f(R)$ with the reduced structure.
The subvariety $K^d_{s,r}$ is called a quantum Grassmannian,
and studied in \cite{Ro}, \cite{SS}, etc.
\end{defn}

\begin{lem}\label{lem_normal&CM}
The quantum Grassmannian $K^d_{s,r}$ is normal and Cohen-Macaulay.
\end{lem}

\begin{proof}
By \cite[Corollary 17]{SS},
the coordinate ring of $K^d_{s,r} \subset  \P\big(\bigwedge^r V \otimes H^0(\calo_{\P^1}(d))^{\vee}\big)$ is normal and Cohen-Macaulay.
Hence this lemma follows.
\end{proof}

\begin{lem}\label{lem_r=0}
For $r=0$,
$f$ is surjective, i.e., $ K^d_{n,0}= \P \big( H^0(\calo_{\P^1}(d))^{\vee} \big)$ holds,
and $f : R \twoheadrightarrow   \P \big( H^0(\calo_{\P^1}(d))^{\vee} \big)$ is a fiber type contraction.
In particular,
$\r+ \alpha$ is an edge of both $\Mov(R)$ and $\Eff(R)$.
\end{lem}

\begin{proof}
By the assumption $0 = r \leq n-2$,
we have $n \geq 2$.
Hence $\dim R =nd > \dim \P \big( H^0(\calo_{\P^1}(d))^{\vee} \big)=d$ holds.
Thus to prove this lemma,
it is enough to show that $f$ is a contraction,
that is,
$f_* \calo_{R}=\calo_{\P( H^0(\calo_{\P^1}(d))^{\vee})}$ holds. 
To show $f_* \calo_{R}=\calo_{\P( H^0(\calo_{\P^1}(d))^{\vee})}$,
it suffices to see that there exists an open subset $U \subset \P \big( H^0(\calo_{\P^1}(d))^{\vee} \big)$
such that $f$ is smooth with irreducible fibers over $U$
since $ \P \big( H^0(\calo_{\P^1}(d))^{\vee} \big)$ is normal.

Set $U=\big\{ [D] \in |\calo_{\P^1}(d)| \, \big| \, D \text{ is reduced} \big\} \subset |\calo_{\P^1}(d)| =\P \big( H^0(\calo_{\P^1}(d))^{\vee} \big)$.
Let $\cald \subset \P^1 \times U$ be the universal divisor,
that is, 
$\cald |_{\P^1 \times \{[D]\}} =D$ holds for each $[D] \in U$.
Then $(\pi_U)_* \calh om (V_{\P^1 \times U}, \calo_{\cald})$ is a locally free sheaf on $|\calo_{\P^1}(d)|$ of rank $nd$.
Hence we obtain the vector bundle
\[
\V_{U} \big((\pi_U)_* \calh om (V_{\P^1 \times U}, \calo_{\cald}) ^{\vee}\big) \arw U,
\]
whose fiber over $[D] \in U$ is the affine space $\Hom (V_{\P^1}, \calo_D) \cong \A^{nd}$.
We write a point in $\V_{U} \big((\pi_U)_* \calh om (V_{\P^1 \times U}, \calo_{\cald}) ^{\vee}\big)$ as $([D], [V_{\P^1} \arw \calo_D])$.
Set an open subset $\widetilde{U}$ by
\[
\widetilde{U} := \big\{ ([D], [V_{\P^1} \arw \calo_D]) \, | \, V_{\P^1} \arw \calo_D \text{ is surjective} \big\} \subset \V_{U} \big((\pi_U)_* \calh om (V_{\P^1 \times U}, \calo_{\cald}) ^{\vee}\big).
\]

\vspace{2mm}
By definition,
$f$ maps a point $[V_{\P^1} \arw B]  \in R$ to
\[
\sum_{P \in \P^1} \big(\mathrm{length}_{k(P)} B \otimes k(P)\big) P \in |\calo_{\P^1}(d)|=\P \big( H^0(\calo_{\P^1}(d))^{\vee}\big).
\]
We note that $B$ is a torsion sheaf since $\rank B= r=0$.
Hence for $[D] \in U$ and $[V_{\P^1} \arw B] \in f^{-1}([D])$,
it holds that $B \cong \calo_D$ since $D$ is reduced.
Thus $(\pi_{f^{-1}(U)})_* \calh om (\calb, (\id_{\P^1} \times f )^* \calo_{\cald})$ is a locally free sheaf on $f^{-1}(U)$ of rank $d$.
Hence we obtain
\[
\V_{f^{-1}(U)} \big((\pi_{f^{-1}(U)})_* \calh om (\calb, (\id_{\P^1} \times f )^* \calo_{\cald})^{\vee} \big) \arw f^{-1}(U),
\]
whose fiber over $[V_{\P^1} \arw B]  \in f^{-1}(U)$ is the affine space $\Hom (B, \calo_D) \cong \A^{d}$,
where $[D]=f([V_{\P^1} \arw B] )$.
Set an open subset $\widetilde{f^{-1}(U)}$ by
\begin{align*}
\lefteqn{\widetilde{f^{-1}(U)} 
 =\big\{ ([V_{\P^1} \arw B], [B \arw \calo_D])  \, | \, B \arw \calo_D \text{ is an isomorphism} \big\}} \hspace{45mm}  \\
 &\subset \V_{f^{-1}(U)} \big((\pi_{f^{-1}(U)})_* \calh om (\calb, (\id_{\P^1} \times f )^* \calo_{\cald})^{\vee}\big).
\end{align*}

\begin{clm}\label{claim_isom}
There exists an isomorphism $\tilde{f} : \widetilde{f^{-1}(U)} \arw \widetilde{U}$ such that
\[
\xymatrix{ 
\widetilde{f^{-1}(U)} \ar[r]^(.6){\tilde{f}} \ar[d] &  \widetilde{U} \ar[d] \\
 f^{-1}(U) \ar[r]^(.6)f & U \\
 }
\]
is commutative.
\end{clm}

\begin{proof}[Proof of Claim \ref{claim_isom}]
We define $\tilde{f}$ by mapping $([V_{\P^1} \stackrel{q}{\arw}B], [B \stackrel{i}{\arw} \calo_D]) \in \widetilde{f^{-1}(U)}$ to
$([D], [V_{\P^1} \stackrel{q}{\arw} B \stackrel{i}{\arw} \calo_D]) \in \widetilde{U}$.
The converse $\widetilde{U} \arw \widetilde{f^{-1}(U)}$ is defined by mapping $ ([D], [V_{\P^1} \stackrel{\bar{q}}{\arw} \calo_D]) \in \widetilde{U}$ to
\[
([V_{\P^1} \stackrel{\bar{q}}{\arw} \calo_D ], [\calo_D \stackrel{\id}{\arw} \calo_D]) \in \widetilde{f^{-1}(U)}.
\]
\end{proof}

Since $\widetilde{f^{-1}(U)} \arw f^{-1}(U)$ and $\widetilde{U} \arw U$ are smooth morphisms with irreducible fibers,
so is $f$ by Claim \ref{claim_isom} and this lemma is proved.
\end{proof}

Assume $r \geq 1$.
Let $R'$ be the Quot scheme parametrizing rank $s$ and degree $d$ quotient sheaves of the dual bundle $V^{\vee}_{\P^1}$ 
and let
\begin{align}\label{eq_univ_on_R'}
0 \arw \cala' \arw V^{\vee}_{\P^1 \times R'} \arw \calb' \arw 0
\end{align}
be the universal exact sequence on $\P^1 \times R'$.
Let $R'^{\circ}  \subset R'$ be the open subset
corresponding to locally free quotient sheaves of $V^{\vee}_{\P^1}$.
Taking the dual of the universal sequence (\ref{eq_univ_on R}) on $\P^1 \times R$,
we have a sequence
\begin{align}\label{eq_dual_of_univ}
0 \arw \calb^{\vee} \arw V^{\vee}_{\P^1 \times R} \arw \cala^{\vee} \arw 0, 
\end{align}
which is exact on $\P^1 \times R^{\circ}$. 
By the universal property of $R'$,
we have a morphism $R^{\circ} \arw R'$.
Similarly,
we have a morphism $R'^{\circ} \arw R$.
By these morphisms, we have an isomorphism $R^{\circ} \cong R'^{\circ}$.
Under this isomorphism,
the restriction of (\ref{eq_dual_of_univ}) on $\P^1 \times R^{\circ}$
coincides with that of (\ref{eq_univ_on_R'}) on $\P^1 \times R'^{\circ}$.

\begin{lem}\label{lem_eff_r=1}
If $r=1$,
$f$ is surjective, i.e., $ K^d_{n-1,1}= \P \big( V \otimes H^0(\calo_{\P^1}(d))^{\vee} \big)$ holds,
and $f : R \twoheadrightarrow \P \big( V \otimes H^0(\calo_{\P^1}(d))^{\vee} \big)$ is a divisorial contraction.
In particular,
$\r+ \alpha$ and $\r+ ( 2d \alpha- \beta)$ are edges of $\Mov(R)$ and $\Eff(R)$ respectively.
\end{lem}

\begin{proof}
When $r=1$,
$R' = \P \big(V \otimes H^0(\calo_{\P^1}(d))^{\vee} \big) $ by \cite[Proposition 6.1 (i)]{St}.
By construction,
the isomorphism $R^{\circ} \cong R'^{\circ}$
is nothing but the restriction of $f : R \arw \P \big(V \otimes H^0(\calo_{\P^1}(d))^{\vee} \big) = R'$ on $R^{\circ}$ in this case.
Thus $f$ is surjective.
Since the Picard numbers of $R$ and $R'$ are two and one respectively,
$f$ is a divisorial contraction.
Hence $\r+ \alpha$ is an edge of $\Mov(R)$.

To compute the class of the contracted divisor,
we compare the canonical divisors on $R$ and $R'$.
Since $r=1$,
\[
K_R=- (n+(2r+2-n)d)\alpha  -(n-2r) \beta = - (n+ (4-n)d) \alpha -(n-2) \beta.
\]
On the other hand,
$K_{R'} = \calo_{R'}(- \dim R' -1) = \calo_{R'}(-nd-n)$
and $f^* \calo_{R'}(1) = \alpha$.
Hence we have
\[
K_R -f^* K_{R'} = (n-2) ( 2d \alpha - \beta).
\]
Since $n-2 \geq r =1$,
the class of the contracted divisor is a positive multiple of $2d \alpha -\beta$
and this class spans an edge of $\Eff(R)$.
\end{proof}

\begin{lem}\label{lem_eff_r_geq2}
If $2 \leq r \leq n-2$,
$R'$ is an SQM of $R$.
Furthermore,
$\Nef(R')= \r+ \alpha +  \r+ ( 2d \alpha -\beta)$
and $\alpha,  2d \alpha - \beta$ are base point free on $R'$.
\end{lem}

\begin{proof}
By \cite[Corollary 4.9]{Sh},
the codimension of $R \setminus R^{\circ} $ in $R$ is $r$,
and that of $R' \setminus R'^{\circ}  $ in $R'$ is $s$.
Hence $R'$ is an SQM of $R$ by $2 \leq r \leq n-2$.
Similar to $R$,
$\Nef(R')$ is spanned by base point free line bundles $\alpha':= c_1(B'_{d}) -c_1(B'_{d-1})$ and $\beta':= c_1(B'_{d-1})$ for  $ B'_m=\pi_* \bigl(\calb'(m)\bigr)$.
Hence it suffices to show that $\alpha'=\alpha$ and $\beta' =  2d \alpha -\beta$
under the identification of $N^1(R)_{\R} $ and $N^1(R')_{\R}$.

Similar to $f$,
there exists a morphism $f' : R' \arw \P \big(\bigwedge^s V^{\vee} \otimes H^0(\calo_{\P^1}(d))^{\vee} \big)$.
By the canonical isomorphism
$ \P \big(\bigwedge^r V \otimes H^0(\calo_{\P^1}(d))^{\vee} \big) 
\stackrel{\sim}{\arw} \P \big(\bigwedge^s V^{\vee} \otimes H^0(\calo_{\P^1}(d))^{\vee} \big)$
and the definitions of $f,f'$,
we have a commutative diagram
\begin{equation}\label{diagram_f_f'}
\begin{gathered}
\xymatrix{
\ar@{}[rd]|{\circlearrowright}
R  \ar[d]_f \ar@{-->}[r]  &  R' \ar[d]^{f'} \\
\P \big(\bigwedge^r V \otimes H^0(\calo_{\P^1}(d))^{\vee} \big) \ar[r]^(.47){\sim} & \P \big(\bigwedge^s V^{\vee} \otimes H^0(\calo_{\P^1}(d))^{\vee} \big).\\
 }
\end{gathered}
\end{equation}
Hence we have $\alpha=f^* \calo(1) = f'^* \calo(1) = \alpha'$.

Let $h= c_1({p_1}^* \calo_{\P^1}(1))$ for $p_1 : \P^1 \times R \arw \P^1$.
As in \cite{St} and \cite{Ma},
we can write the Chern classes $c_i(\cala) = t_i + h u_{i-1}$ for $i=1,2$, $t_i \in A^{i}(R), u_{i-1} \in A^{i-1}(R)$.
By \cite[Lemma 3.1]{Ma},
$t_1= -\alpha, u_1=\beta$.
In particular,
$\beta = {\pi_R }_* (c_2(\cala))$ holds.
Similarly,
we have $\beta'= {\pi_{R'} }_* (c_2(\cala'))$.
Under the natural identifications $A^1(R') = A^1(R'^{\circ} ) =  A^1({R}^{\circ} ) =A^1(R)$,
we have
\begin{align*}
\beta' = {\pi_{R'}}_* (c_2(\cala ')) &= {\pi_{{R'}^{\circ}}}_* (c_2(\cala ' |_{\P^1 \times {R'}^{\circ}}))\\
&=  {\pi_{{R}^{\circ}}}_* (c_2(\calb^{\vee} |_{\P^1 \times {R}^{\circ}}) )= {\pi_{{R}^{\circ}}}_* (c_2(\calb |_{\P^1 \times {R}^{\circ}}))
= {\pi_{R}}_* (c_2(\calb ))
\end{align*}
As written in the proof of  \cite[Lemma 3.1]{Ma},
it holds that $c_2(\calb)= t_1^2 - t_2 -h (2dt_1+u_1)$.
Hence we have $\beta' = {\pi_{R}}_* (c_2(\calb )) = -(2dt_1+u_1) = 2d \alpha - \beta$.
\end{proof}

\begin{rem}\label{rem_D&Y}
To describe $\Eff(R)$,
Jow \cite{Jo} used another basis $Y,D$
 of $\Pic(R)$,
which was introduced by Mart\'inez in \cite{Ma}. 
The classes $Y$ and $D$ are defined by
\[
Y = {\pi_R}_* (h \cdot c_1(\calb)) , \quad D = {\pi_R}_* (c_2(\calb)).
\]
By Lemma 3.2 in \cite{Ma} and Introduction in \cite{Jo},
$Y=\alpha$ and $D= 2d \alpha -\beta$ holds.
Hence $Y$ and $D$ are nothing but $\alpha'$ and $\beta'$ respectively if $2 \leq r \leq n-2$.
\end{rem}

As in Remark \ref{rem_small_contraction},
$\pr_1,\pr_2$ in the diagram (\ref{diagram_small_contractions}) are small contractions for $\lfloor d/s \rfloor +1 \leq m \leq d-1$.
In the following lemma,
we see that $f,f'$ in the diagram (\ref{diagram_f_f'}) are small contractions.
We note that $f'(R') =f(R) =K^d_{s,r}$ holds by the diagram (\ref{diagram_f_f'}) for $2 \leq r \leq n-2$
under the identification $ \P \big(\bigwedge^r V \otimes H^0(\calo_{\P^1}(d))^{\vee} \big) 
= \P \big(\bigwedge^s V^{\vee} \otimes H^0(\calo_{\P^1}(d))^{\vee} \big)$.

\begin{lem}\label{lem_quantum_Gr}
If $2 \leq r \leq n-2$,
$f : R \twoheadrightarrow K^d_{s,r}$ and $f' : R' \twoheadrightarrow K^d_{s,r}$ are small contraction.
\end{lem}

\begin{proof}
It suffices to show that $f : R \twoheadrightarrow K^d_{s,r}$ is a small contraction.
By \cite[Corollary 4.9]{Sh},
the codimension of $R \setminus R^{\circ}  $ in $R$ is $r \geq 2$.
Since $K^d_{s,r}$ is normal by Lemma \ref{lem_normal&CM},
it is enough to show that $f$ is an embedding on $R^{\circ}$.

Since $\Gr(V,r)$ is embedded into $ \P(\bigwedge^r V)$ by the Pl\"{u}cker embedding,
$R^{\circ}=\Mor_d(\P^1, \Gr(V,r)) $ is embedded into $\Mor_d \big(\P^1,  \P(\bigwedge^r V) \big) $.
Let $\bar{R} $ be the Quot scheme parametrizing all rank $1$, degree $d$ quotient sheaves of $(\bigwedge^r V)_{\P^1}$,
which is a compactification of $\Mor_d \big(\P^1,  \P(\bigwedge^r V) \big) $ since $\P(\bigwedge^r V) =\Gr(\bigwedge^r V,1)$.
Applying the definition of $f$ to $\bar{R}$,
we have a morphism $\bar{f} : \bar{R} \arw  \P \big( \bigwedge^r V \otimes H^0(\calo_{\P^1}(d))^{\vee} \big)$.
By the constructions of $f, \bar{f}$,
the diagram
\[
\xymatrix{
R^{\circ}=\Mor_d(\P^1, \Gr(V,r))  \ar[rd]_{f |_{R^{\circ}}} \ar@{^(->}[r]  &  \Mor_d \big(\P^1,  \P(\bigwedge^r V) \big) \subset \bar{R} \ar[d]^{\bar{f} |_{\Mor_d (\P^1,  \P(\bigwedge^r V))}}\\
& \P \big(\bigwedge^r V \otimes H^0(\calo_{\P^1}(d))^{\vee} \big) .\\
 }
\]
is commutative.
Applying the proof of Lemma \ref{lem_eff_r=1} to $\bar{R}$,
the restriction $\bar{f} |_{\Mor_d (\P^1,  \P(\bigwedge^r V))}$ is an open immersion.
Hence $f |_{R^{\circ}}$ is an embedding and this lemma holds.
\end{proof}

\begin{lem}\label{lem_dual}
Assume $2 \leq r \leq n-2$.
For $\lfloor d/r \rfloor +1 \leq  m' \leq d-1$,
there exists a smooth SQM $R'_{m'}$ of $R$
such that
\[
\Nef(R'_{m'})= \r+ ((d+m')\alpha - \beta) + \r+ ((d+m'+1)\alpha - \beta)
\]
and $(d+m')\alpha- \beta $ and $(d+m'+1)\alpha-\beta $ are base point free on $R'_{m'} $.
Furthermore,
$\r+ ((d+ \lfloor d/r \rfloor +1) \alpha) -\beta )$ and $\r+ ((d+ \lceil d/r \rceil ) \alpha -\beta ) $
are edges of $\Mov(R)$ and $\Eff(R)$ respectively.
\end{lem}

\begin{proof}
Applying Proposition \ref{lem_SQM} to $R'$,
we obtain SQMs $R'_{m'}$ of $R'$ such that
\[
\Nef(R'_{m'}) =\r+  (-(d-1 -m')\alpha' +\beta') + \r+ ( - (d-m')\alpha' +\beta').
\]
Since $\alpha'=\alpha$ and $\beta'= 2d \alpha-\beta$ by the proof of Lemma \ref{lem_eff_r_geq2},
we have the first assertion.
The rest part follows by applying Lemma \ref{lem_edge_Mov_Eff} to $R'$.
\end{proof}

\begin{proof}[Proof of Theorem \ref{main thm1}]
(1) follows from Lemma \ref{lem_SQM}.
(2) follows from Lemmas \ref{lem_edge_Mov_Eff}, \ref{lem_r=0}, and \ref{lem_eff_r=1}.
(3) follows from Lemmas \ref{lem_r=0}, and \ref{lem_eff_r=1}.
(4) follows from Lemma \ref{lem_edge_Mov_Eff}.
(5) follows from Remark \ref{rem_small_contraction}.
\end{proof}

\begin{proof}[Proof of Theorem \ref{main thm2}]
(1) follows from Lemmas \ref{lem_SQM},  \ref{lem_eff_r_geq2}, \ref{lem_dual}.
(2) follows from Lemmas \ref{lem_edge_Mov_Eff}, \ref{lem_dual}.
(3) follows from Lemma \ref{lem_edge_Mov_Eff}.
(4) follows by applying Lemma \ref{lem_edge_Mov_Eff} to $R'$.
(5) follows from Remark \ref{rem_small_contraction}.
(6) follows by applying Remark \ref{rem_small_contraction} to $R'$.
(7) follows from Lemma \ref{lem_quantum_Gr}.
\end{proof}

\begin{proof}[Proof of Corollary \ref{cor_MDS}]\label{proof_log_fano}
We may assume that $0 \leq r \leq n-2$ and $d \geq 1$.

By results of Str\o mme,
$R$ satisfies conditions 1) and 2) in the definition of Mori dream spaces.
By Theorems \ref{main thm1}, \ref{main thm2},
$R$ satisfies condition 3).
Hence $R$ is a Mori dream space.

By Theorems \ref{main thm1}, \ref{main thm2} and the description of $K_R$ by $\alpha,\beta$,
we can check that $-K_R \in \Mov(R) \cap \Eff(R)^{\circ}$,
where $\Eff(R)^{\circ} $ is the interior of $\Eff(R)$.
Hence there exists an SQM $R^{\dagger}$ of $R$ such that whose anti-canonical divisor is nef and big.
Since $R^{\dagger}$ is smooth,
$R^{\dagger}$ is log Fano.
Hence $R$ is log Fano by \cite[Lemma 2.4]{Bi}. 
\end{proof}


\begin{thebibliography}{Kol}


\bibitem[ACGH]{ACGH}
E. Arbarello, M. Cornalba, P.A. Griffiths, and J. Harris,
\emph{Geometry of algebraic curves. Vol. I.},
Grundlehren der Mathematischen Wissenschaften [Fundamental Principles of Mathematical Sciences], \textbf{267}. Springer-Verlag, New York, 1985. xvi+386 pp.



\bibitem[ABCH]{ABCH}
D. Arcara, A. Bertram, I. Coskun, and J. Huizenga,
\emph{The minimal model program for the Hilbert scheme of points on $\P^2$ and Bridgeland stability},
Adv. Math. \textbf{235} (2013), 580--626. 


\bibitem[BM]{BM}
A. Bayer and E. Macr\`i,
\emph{Projectivity and birational geometry of Bridgeland moduli spaces},
J. Amer. Math. Soc. \textbf{27} (2014), no. 3, 707--752. 



\bibitem[Bi]{Bi}
C. Birkar,
\emph{Singularities on the base of a Fano type fibration},
to appear in J.\ Reine Angew.\ Math.
doi:10.1515/crelle-2014-0033.




\bibitem[BCHM]{BCHM}
C. Birkar, P. Cascini, C. Hacon, and J. M\textsuperscript{c}Kernan,
\emph{Existence of minimal models for varieties of log general type},
J. Amer. Math. Soc. \textbf{23} (2010), no. 2, 405--468.




\bibitem[Ch]{Ch}
D. Chen,
\emph{Mori's program for the Kontsevich moduli space $\overline{M}_{0,0}(\P^3,3)$},
Int. Math. Res. Not. IMRN 2008, Art. ID rnn 067, 17 pp.



\bibitem[Gr]{Gr}
A. Grothendieck,
\emph{Techniques de construction et th\'eor\`emes d'existence en g\'eom\'etrie alg\'ebrique. IV. Les sch\'emas de Hilbert},
S\'eminaire Bourbaki, Vol. 6, Exp. No. 221, 249--276, Soc. Math. France, Paris, 1995.





\bibitem[Ha]{Ha}
B. Hassett,
\emph{Classical and minimal models of the moduli space of curves of genus two},
Geometric methods in algebra and number theory, 169--192, 
Progr. Math., 235, Birkh\"auser Boston, Boston, MA, 2005. 


\bibitem[HK]{HK}
Y. Hu and S. Keel, 
\emph{Mori dream spaces and GIT}, 
Michigan Math. J. \textbf{48} (2000), 331--348. 






\bibitem[Jo]{Jo}
S-Y. Jow, 
\emph{The effective cone of the space of parametrized rational curves in a Grassmannian}, 
Math. Z. \textbf{272} (2012), no. 3--4, 947--960.




\bibitem[Ma]{Ma}
C. Mart\'inez, 
\emph{On a stratification of the Kontsevich moduli space $\overline{M}_{0,n}(G(2,4),d) $ and enumerative geometry}, 
J. Pure Appl. Algebra \textbf{213} (2009), no. 5, 857--868.





\bibitem[Ni]{Ni}
N. Nitsure,
\emph{Construction of Hilbert and Quot schemes},
Fundamental algebraic geometry, 105--137, Math. Surveys Monogr., 123, Amer. Math. Soc., Providence, RI, 2005.




\bibitem[Ro]{Ro}
J. Rosenthal,
\emph{On dynamic feedback compensation and compactification of systems},
SIAM J. Control Optim. \textbf{32} (1994), no. 1, 279--296.



\bibitem[Sh]{Sh}
Y. Shao,
\emph{A Compactification of the Space of Parametrized Rational Curves in Grassmannians},
arXiv:1108.2299.



\bibitem[SS]{SS}
F. Sottile and B. Sturmfels,
\emph{A sagbi basis for the quantum Grassmannian},
J. Pure Appl. Algebra \textbf{158} (2001), no. 2--3, 347--366.


\bibitem[St]{St}
S. Str\o mme,
\emph{On parametrized rational curves in Grassmann varieties},
Space curves (Rocca di Papa, 1985), 251--272,
Lecture Notes in Math., 1266, Springer, Berlin, 1987. 



\bibitem[Ve]{Ve}
K. Venkatram,
\emph{Birational Geometry of the Space of Rational Curves in Homogeneous Varieties},
Thesis (Ph.D.)--Massachusetts Institute of Technology. 2011.



\end{thebibliography}
\end{document}